\documentclass[10pt, reqno]{amsart}
\usepackage{amsmath,amssymb,amsthm,graphicx}
\usepackage[left=2.5cm,right=2.2cm,top=2.5cm,bottom=3.5cm]{geometry}

\usepackage[breaklinks,hypertexnames=false]{hyperref}
\hypersetup{
	colorlinks = true, 
	urlcolor = blue, 
	linkcolor = blue, 
	citecolor = blue 
}
\usepackage{orcidlink}
\usepackage{dsfont}
\usepackage{parskip}
\makeatletter
\def\thm@space@setup{%
  \thm@preskip=\parskip \thm@postskip=0pt
}
\makeatother

\usepackage[hang,flushmargin]{footmisc}
\usepackage{footnote} 
\makesavenoteenv{tabular}

\usepackage{caption}
\let\le\leqslant
\let\ge\geqslant
\let\a\alpha

\let\eps\varepsilon

\let\mc\mathcal
\let\mb\mathbb
\newcommand{\ZZ}{\mathbb{Z}}
\newcommand{\RR}{\mathbb{R}}
\newcommand{\NN}{\mathbb{N}}

\usepackage{thmtools}
\usepackage[noabbrev,capitalize,nameinlink,nosort]{cleveref}

\newtheorem{lemma}{Lemma}[section]
\newtheorem{proposition}[lemma]{Proposition}
\newtheorem{theorem}[lemma]{Theorem}
\newtheorem{corollary}[lemma]{Corollary}

\crefname{fact}{Fact}{Facts}
\newtheorem{conjecture}[lemma]{Conjecture}

\theoremstyle{definition}
\newtheorem{definition}[lemma]{Definition}

\newtheorem{remark}[lemma]{Remark}
\newtheorem{claim}[lemma]{Claim}
\crefname{claim}{Claim}{Claims}
\newtheorem*{remark*}{Remark}

\DeclareMathOperator{\ind}{ind}
\DeclareMathOperator{\Ber}{Ber}
\DeclareMathOperator{\Slice}{Slice}
\DeclareMathOperator{\Bin}{Bin}
\DeclareMathOperator{\Poi}{Poi}
\DeclareMathOperator{\LO}{LO}

\newcommand{\prob}[1]{\mathbb{P}[#1]}
\newcommand{\probs}[2]{\mathbb{P}_{#1}[#2]}
\newcommand{\expected}[1]{\mathbb{E}[#1]}
\newcommand{\expecteds}[2]{\mathbb{E}_{#1}[#2]}

\let\originalleft\left
\let\originalright\right
\renewcommand{\left}{\mathopen{}\mathclose\bgroup\originalleft}
\renewcommand{\right}{\aftergroup\egroup\originalright}

\crefformat{equation}{#2(#1)#3}

\newcommand*{\claimproofname}{Proof of claim}
\newenvironment{claimproof}[1][\claimproofname]{\begin{proof}[#1]}{\end{proof}}

\newcommand{\binmax}[2]{\mathrm{MP}(#1,#2)}
\newcommand{\binmaxplus}[2]{\mathrm{MP}_+(#1,#2)}

\title[]{\huge E\lowercase{dge-statistics beyond }$1/e$}
\author[]{\Large A\lowercase{lexandr }G\lowercase{rebennikov and }M\lowercase{atthew }K\lowercase{wan}}

\address{Institute of Science and Technology Austria (ISTA)}
\email{\href{mailto:aleksandr.grebennikov@ist.ac.at}{\nolinkurl{aleksandr.grebennikov@ist.ac.at}}}
\address{Institute of Science and Technology Austria (ISTA)}
\email{\href{mailto:matthew.kwan@ist.ac.at}{\nolinkurl{matthew.kwan@ist.ac.at}}}

\thanks{Both authors are supported by ERC Starting Grant “RANDSTRUCT” No. 101076777.}

\begin{document}

\begin{abstract}
For integers $k$ and $\ell$, let $\ind(k, \ell)$ be the maximum proportion of $k$-vertex subsets of a large graph that induce exactly $\ell$ edges. The edge-statistics theorem (conjectured by Alon--Hefetz--Krivelevich--Tyomkyn, and proved by Kwan--Sudakov--Tran, Fox--Sauermann, and Martinsson--Mousset--Noever--Truji\'c) asserts that, for $k \to \infty$ and $0 < \ell <\binom{k}{2}$, one has $\ind(k, \ell) \le 1/e + o(1)$.

We investigate the ``stability'' of this problem: how can one improve this bound under additional assumptions on $\ell$? In particular, the edge-statistics theorem is tight when $\ell\in \{1,k-1,\binom k2-(k-1),\binom k2-1\}$; we show that for all other $\ell$, one can replace $1/e$ with a strictly smaller constant. This extends an analogous result of Ueltzen in the setting of graph inducibility. We also obtain a much stronger (and essentially optimal) upper bound on $\ind(k, \ell)$ when $\ell$ is far from a multiple of $k$, refining and extending previous bounds due to Fox and Sauermann.
\end{abstract}

\maketitle

\section{Introduction}
\label{sec:introduction}

For a $k$-vertex graph $H$, let $N(n, H)$ be the maximum possible number of $k$-vertex subsets of an $n$-vertex graph that induce a copy of $H$. A simple averaging argument shows that $N(n, H) / \binom{n}{k}$ is non-increasing in $n$, and thus one can define the \emph{inducibility} of $H$ as
\[
\ind(H) = \lim_{n \to \infty} \frac{N(n, H)}{\binom{n}{k}}.
\]
This concept was introduced by Pippenger and Golumbic~\cite{pippenger-golumbic-75} in 1975, and has been extensively studied since then.

Alon, Hefetz, Krivelevich, and Tyomkyn~\cite{alon-hefetz-krivelevich-tyomkyn-20} introduced the following variant of this notion. Let $N(n, k, \ell)$ be the maximum number of $k$-vertex subsets of an $n$-vertex graph that induce exactly $\ell$ edges. Similarly, $N(n, k, \ell) / \binom{n}{k}$ is non-increasing in $n$, and one can define the \emph{edge-inducibility} as
\[
\ind(k, \ell) = \lim_{n \to \infty} \frac{N(n, k, \ell)}{\binom{n}{k}}.
\]
Clearly, for each $k$ we have $\ind(k, 0) = \ind(k, \binom{k}{2}) = 1$ (since we can take the host graph to be empty or complete). Also note that $\ind(k, \ell) = \ind(k, \binom{k}{2} - \ell)$, since one can replace the host graph with its complement. Thus, throughout the rest of this introduction we restrict our attention to the range $1 \le \ell \le \frac{1}{2}\binom{k}{2}$.

It follows from Goodman's theorem that $\ind(3, 1) = 3/4$. However, in general it is surprisingly difficult to determine edge-inducibilities exactly. For several small values of $k$ and $\ell$, this was recently done by Bodn\'ar and Pikhurko~\cite{bodnar-pikhurko-25} via the method of flag algebras. Also, in this direction, Liu, Mubayi, and Reiher~\cite[Theorem 1.13]{liu-mubayi-reiher-23} obtained an explicit formula for $\ind(k, 1)$, for all $k$.

Alon, Hefetz, Krivelevich, and Tyomkyn suggested studying the asymptotic behaviour of edge-inducibilities as $k \to \infty$, and posed the following \emph{edge-statistics conjecture} \cite[Conjecture 1.2]{alon-hefetz-krivelevich-tyomkyn-20}: if $k$ is sufficiently large in terms of $\eps > 0$ and $1 \le \ell \le \frac{1}{2}\binom{k}{2}$, then
\begin{equation} \label{eq:edge-statistics}
\ind(k, \ell) \le 1/e + \eps.
\end{equation}
This conjecture was proved in a combination of works by Kwan, Sudakov, and Tran~\cite{kwan-sudakov-tran-19}, Fox and Sauermann~\cite{fox-sauermann-20}, and Martinsson, Mousset, Noever, and Truji\'c~\cite{martinsson-mousset-noever-trujic-19}, and was recently extended to hypergraphs by Jain, Kwan, Mubayi, and Tran~\cite{jain-kwan-mubayi-tran-25}.

To see that the bound \cref{eq:edge-statistics} is asymptotically sharp, one can consider $\ell = 1$ or $\ell = k-1$. Indeed, suppose that $k \to \infty$, and that $n$ is much larger than $k$. Then,
\begin{itemize}
    \item for $\ell = 1$, consider the random graph $\mb G(n, \binom{k}{2}^{-1})$: the number of edges in its random $k$-vertex subgraph converges to a Poisson random variable with mean $1$;
    \item for $\ell = k-1$, consider the complete bipartite graph $K_{n/k, n - n/k}$: the size of the intersection of a random $k$-subset of its vertices with the smaller part converges to a Poisson random variable with mean $1$.
\end{itemize}

In this paper we investigate the ``stability'' of this problem: how can one improve the bound \cref{eq:edge-statistics} under additional assumptions on $\ell$? Our first result states that for $\ell \notin \{1, k-1\}$, we can replace $1/e \approx 0.37$ with a strictly smaller constant.

\begin{theorem} \label{better-than-1/e}
    Let $k, \ell$ be positive integers such that $k$ is sufficiently large, and $1 \le \ell \le \frac{1}{2}\binom{k}{2}$, $\ell \notin \{1, k-1\}$. Then, 
    \[
    \ind(k, \ell) < 0.33.
    \]
\end{theorem}

This extends\footnote{Strictly speaking, our result as stated does not provide a bound on $\ind(H)$ when $H$ has exactly $k-1$ edges but is not a star. However, it follows from our proof that for such graphs $\ind(H) \le 2/e^2 + o(1)$.} a recent result of Ueltzen~\cite[Theorem 1.2]{ueltzen-24} in the setting of graph inducibility, which states that for every $k$-vertex graph $H$ such that $2 \le e(H) \le \frac{1}{2}\binom{k}{2}$ and $H \not \simeq K_{1, k-1}$, one has $\ind(H) \le c + o(1)$ as $k \to \infty$ for some $c < 1/e$. Also, \cref{better-than-1/e} confirms the uniformity-two case of a general conjecture due to Jain, Kwan, Mubayi, and Tran~\cite{jain-kwan-mubayi-tran-25} on edge-inducibilities in hypergraphs (see \cref{subsec:hypergraphs}).

\subsection{Strong bounds when \texorpdfstring{$\ell$}{} is far from a small multiple of \texorpdfstring{$k$}{}} 

Another interesting direction is to study assumptions on $\ell$ that allow one to prove bounds of the form $\ind(k, \ell) = o(1)$ as $k \to \infty$. For the ``dense range'', i.e. when $\ell/k\to \infty$, an essentially optimal such bound was obtained by Kwan and Sauermann~\cite{kwan-sauermann-23} (building upon the previous work of Kwan, Sudakov, and Tran~\cite[Theorem 1.1]{kwan-sudakov-tran-19} and Alon, Hefetz, Krivelevich, and Tyomkyn~\cite[Theorem 1.5]{alon-hefetz-krivelevich-tyomkyn-20}).

\begin{theorem}[{Kwan--Sauermann~\cite[Theorem 1.3]{kwan-sauermann-23}}] \label{bulk}
    Let $k, \ell$ be positive integers such that $1 \le \ell \le \frac{1}{2}\binom{k}{2}$. Then, 
    \[
    \ind(k, \ell) = O\left(\sqrt{k/\ell}\right).
    \]
\end{theorem}

However, for $\ell = O(k)$ this bound is vacuous. Fox and Sauermann~\cite[Theorem 1.3]{fox-sauermann-20} studied edge-inducibilities in the ``very sparse'' regime $\ell \le k / 2$: they obtained an optimal bound of $O(\ell^{-1/4})$ for $\ell \le k/\log^4 k$, and a slightly suboptimal bound of $O(k^{-1/4} \log k)$ for $k/\log^4 k \le \ell \le (1/2 - o(1))k$. 

In our next result, \cref{far-from-multiple}, we refine these estimates by showing that the optimal bound $O(\ell^{-1/4})$ holds for all $\ell \le k - \sqrt{k}$. In fact, more generally, we prove that an analogous statement is true for all values of $\ell$ that are far from a multiple of $k$.

\begin{theorem} \label{far-from-multiple}
    Let $k, \ell_0, a$ be non-negative integers such that $1 \le \ell_0 \le k-1$ and $0 \le a \le \sqrt{k}$. Then,
    \[
    \ind(k, a k + \ell_0) = O\left(\max(\ell_0^{-1/4}, ((k-\ell_0)/(a+1))^{-1/2})\right).
    \]
\end{theorem}

In the range $\ell = O(k)$, this result is, in some sense, optimal (see \cref{subsec:lower-bounds} for a further discussion). On the other hand, while the bound from \cref{far-from-multiple} still holds for $a > \sqrt{k}$, in this range it gets superseded by the bound from \cref{bulk}.

\begin{figure}
\begin{center}
\includegraphics[width=0.9\textwidth]{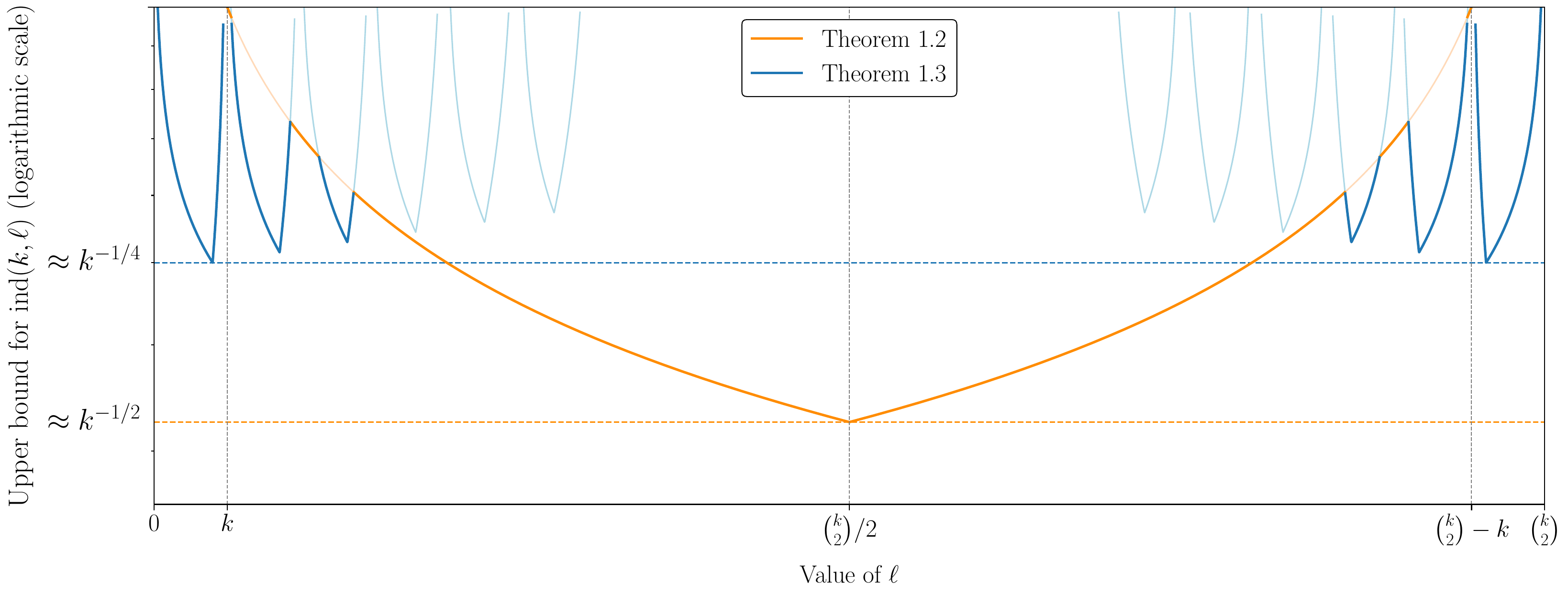}
\end{center}
\caption{This cartoon plot summarises the upper bounds on $\ind(k, \ell)$ given by \cref{bulk,far-from-multiple}.} 
\end{figure}

\subsection{Asymptotic results when \texorpdfstring{$\ell$}{} is close to a small multiple of \texorpdfstring{$k$}{}}

Together, \cref{bulk,far-from-multiple} imply that, for every $\eps > 0$, we have $\ind(k, \ell) \le \eps$ unless $\ell = ak + b$ with $0 \le a \le C$ and $-C \le b \le C$ for some $C = C(\eps)$. That is to say,  $\ind(k, \ell)$ is small unless $\ell$ is close to a multiple of $k$. In the range where $\ell$ is close to a multiple of $k$, one cannot hope for a bound that goes to zero as $k$ tends to infinity.
Indeed, taking the host graph to be the complete bipartite graph $K_{an/k, n - an/k}$, or, respectively, $K_{an/k, n-an/k}$ with additional edges between each pair of vertices in the smaller part, it is not hard to see that for each fixed $a \ge 1$ and $k \to \infty$,
\begin{equation} \label{eq:lower-bound-a^a/(e^a a!)}
\ind(k, a(k-a)) \ge \frac{a^a}{e^a a!} + o(1), \qquad \ind(k, a(k-a)+{\textstyle{\binom{a}{2}}}) \ge \frac{a^a}{e^a a!} + o(1).
\end{equation}
Our next result provides an upper bound that is sharp in the above cases. Moreover, for values of $\ell$ that do not fall in an interval of the form $[a(k-a), a(k-a)+\binom{a}{2}]$, we also obtain an asymptotic improvement over this bound.

\begin{theorem} \label{close-to-multiple}
    Fix any integer $a \ge 1$ and $C, \eps > 0$, and suppose that $k$ is sufficiently large in terms of $a, C, \eps$. Consider an integer $\ell \in [a(k-a) - C, a(k-a) + C]$.
    \begin{enumerate}
        \item[(1)] Then, 
        \[
        \ind(k, \ell) \le \frac{a^a}{e^a a!} + \eps;
        \]
        \item[(2)] If, additionally, $\ell$ does \textbf{not} satisfy $a(k-a) \le \ell \le a(k-a) + \binom{a}{2}$, then
        \[
        \ind(k, \ell) \le \frac{(a+1)^{a+1}}{e^{a+1} (a+1)!} + \eps.
        \]
    \end{enumerate}
\end{theorem}
In particular, combining the lower bounds \cref{eq:lower-bound-a^a/(e^a a!)} with the upper bounds given by \cref{close-to-multiple}(1), we conclude that both $\ind(k, a(k-a))$ and $\ind(k, a(k-a)+{\textstyle{\binom{a}{2}}})$ tend to $a^a/(e^a a!)$ as $k \to \infty$.

Since $M^M/(e^M M!)$ is a decreasing function of $M \in \NN$, the bounds given by \cref{close-to-multiple} for $\ell \neq k-1$ are at most $2/e^2 + \eps$. Combined with \cref{bulk,far-from-multiple}, this implies that $\ind(k, \ell) \le 2/e^2 + o(1)$ unless $\ell = O(1)$. We make this explicit in the following theorem.

\begin{theorem} \label{2/e^2}
    Fix $\eps > 0$, and suppose that $k$ is sufficiently large in terms of $\eps$. Then, for each $\ell$ such that $60 \le \ell \le \frac{1}{2}\binom{k}{2}$ and $\ell \neq k-1$, 
    \[
    \ind(k, \ell) \le 2/e^2 + \eps.
    \]
\end{theorem}

We believe that in fact the above statement should hold with $60$ replaced by $2$, thereby improving the bound in \cref{better-than-1/e} from $0.33$ to the optimal $2/e^2 + o(1) \approx 0.27$. However, we did not manage to accomplish this (see \cref{subsec:set-systems} for some speculations in this direction). A related conjecture in the setting of graph inducibility was proposed by Ueltzen~\cite[Conjecture 1.6]{ueltzen-24}: it states for each graph $H$ with $k$ vertices such that $2 \le e(H) \le \frac{1}{2}\binom{k}{2}$ and $H \not \simeq K_{1, k-1}$, one has $\ind(H) \le 2/e^2 + o(1)$ as $k \to \infty$. Ueltzen proved this conjecture for graphs with $\Omega(k)$ non-isolated vertices \cite[Theorem 2.2]{ueltzen-24}, and \cref{2/e^2} can be used to confirm it for every graph with at least $60$ edges.

\begin{figure}
\begin{center}
\includegraphics[width=0.8\textwidth]{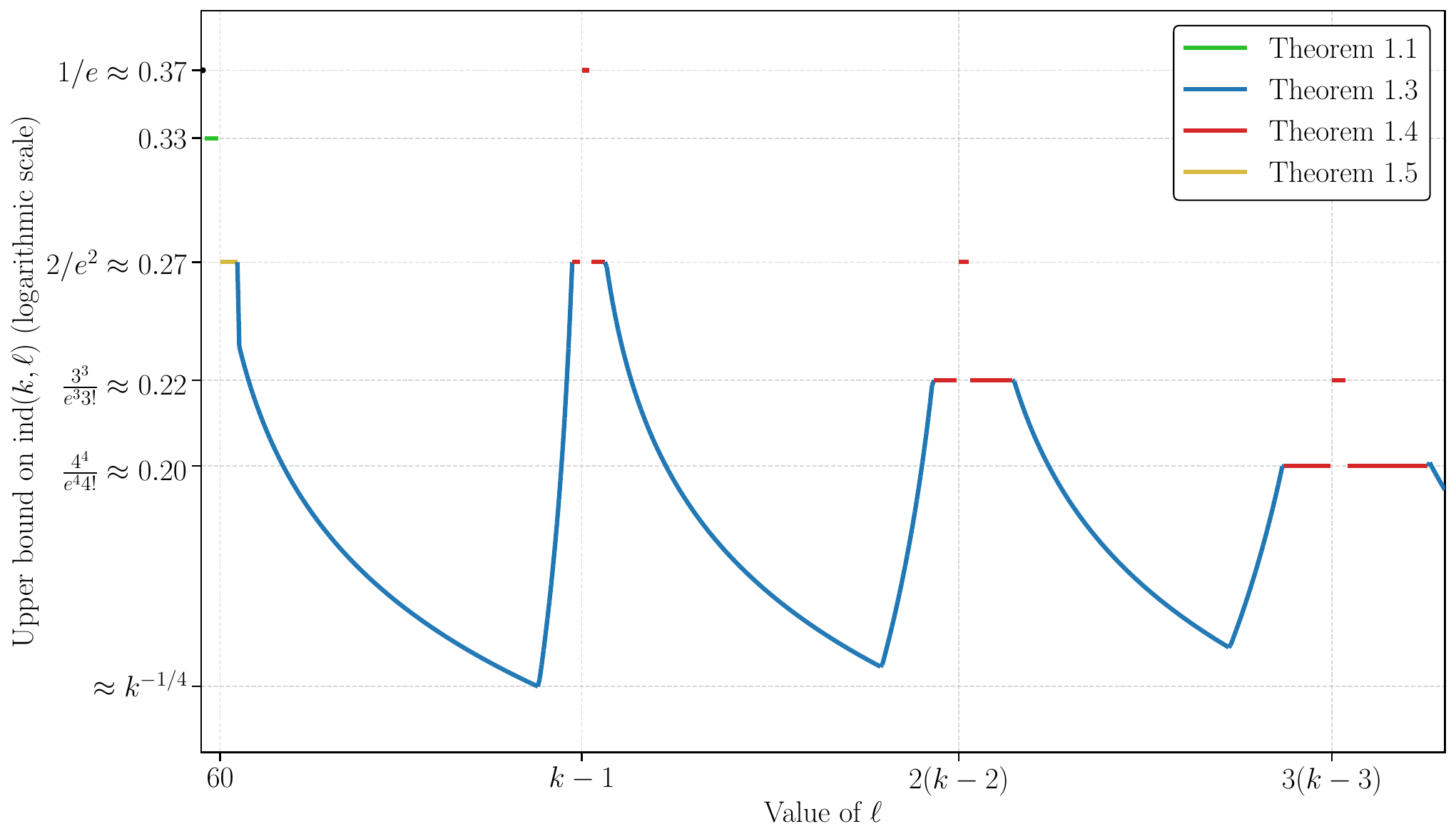}
\end{center}
\caption{This cartoon plot summarises the upper bounds on $\ind(k, \ell)$ in the range $\ell = O(k)$ and $k \to \infty$ given by \cref{better-than-1/e,far-from-multiple,close-to-multiple,2/e^2}.}
\end{figure}

\subsection{Edge-inducibilities in uniform hypergraphs}
\label{subsec:hypergraphs}

All the questions discussed in this introduction also make sense in the setting of hypergraphs. Formally, for each integer $r \ge 2$, define $\ind_r(k, \ell)$ as the limit of $N_r(n, k, \ell) / \binom{n}{k}$ as $n \to \infty$, where $N_r(n, k, \ell)$ is the maximum number of $k$-vertex subsets of an $n$-vertex $r$-uniform hypergraph that induce exactly $\ell$ edges. Jain, Kwan, Mubayi, and Tran~\cite{jain-kwan-mubayi-tran-25} proved that $\ind_r(k, \ell) \le 1/e + o(1)$ for each $\ell \notin \{0, \binom{k}{r}\}$, where the asymptotics is for $k \to \infty$ with $r$ fixed. They also conjectured that $1/e$ can be replaced by a smaller constant unless $\min(\ell, \binom{k}{r} - \ell)$ is of the form $\binom{k-d}{r-d}$ for some $d \in \{0, \ldots, r-1\}$. \cref{better-than-1/e} confirms this conjecture for $r = 2$.

While our proof of \cref{better-than-1/e} is quite specific to the graph case, some of the ideas behind \cref{bulk,far-from-multiple} can be extended to hypergraphs. We combine our approach with the methods from \cite{jain-kwan-mubayi-tran-25} (partially based on results of Bollob\'as and Scott~\cite{bollobas-scott-15}) to prove the following weaker analogue of these statements.
\begin{theorem} \label{hypergraphs-weak}
    For every integer $r \ge 2$ and $\eps > 0$, there exists $C = C(\eps, r)$ such that if $\ind_r(k, \ell) > \eps$ then
    \[
    \ell \in \Big\{\sum_{d = 0}^r \ell_d \binom{k-d}{r-d} : \ell_0 \in \{0, 1\}, \; \ell_1, \ldots, \ell_r \in \ZZ \cap [-C, C]\Big\}.
    \]
\end{theorem}
Roughly speaking, \cref{hypergraphs-weak} says that $\ind_r(k, \ell)$ is negligible unless $\ell$ is a linear combination of terms of the form $\binom {k-d}{r-d}$ with small coefficients (it is not hard to see that the converse is also true). The quantitative aspects of the dependence of $C$ on $\eps$ given by our proof, however, appear to be non-optimal.

\subsection{Proof ideas} 

First, we briefly outline the proof of \cref{far-from-multiple}. Since $\ind(k, \ell)$ is the limit of a non-increasing sequence $N(n, k, \ell) / \binom{n}{k}$, it suffices to prove upper bounds on $N(n, k, \ell)$ for $n = 3k$. This, in turn, boils down to estimating the probability that a uniformly random $k$-subset $X$ of an $n$-vertex graph $G$ contains exactly $\ell$ edges. Since in our case $\ell = o(k^2)$, we may assume that $G$ is sparse (otherwise, the bound easily follows from a standard concentration inequality).

If $G$ has a large matching, then we obtain the desired bound by combining a recent anticoncentration result for sparse polynomials of a uniformly random point on a ``slice'' of the Boolean hypercube (due to Jain, Kwan, Mubayi, and Tran~\cite{jain-kwan-mubayi-tran-25}) with the optimal bound for the quadratic Littlewood--Offord problem (due to Kwan and Sauermann~\cite{kwan-sauermann-23}). Specifically, we deal with this case in \cref{large-matching-corollary}.

If $G$ has no large matching, then it has a small vertex cover $U$. Conditioning on the outcome of $U \cap X$, we reduce our problem to a question about a \emph{linear} polynomial on the ``slice''. If almost all its coefficients are the same, then its value cannot be far from a multiple of $k$, and thus cannot be equal to $\ell$. Otherwise, we can bound its point concentration probability using a variant of the classical Erd\H{o}s--Littlewood--Offord theorem (due to Litvak, Lytova, Tikhomirov, Tomczak-Jaegermann, and Youssef~\cite{LLTTY-17}).

The proof of \cref{close-to-multiple} also begins with an application of \cref{large-matching-corollary}, but the rest of the argument is more delicate. Namely, instead of considering the whole vertex cover $U$, we find a subset $W \subseteq U$ of ``important'' vertices (such that the degree of each vertex in $W$ is linear in $n$ and is much larger than the total number of edges not touching $W$). Next, we prove that the number of edges inside $X$ is very unlikely to be equal to $\ell$ unless $W \cap X$ falls into a certain \emph{antichain} of subsets of $W$. Using the classical BLYM inequality on antichains, along with a result of Ehm~\cite{ehm-91} comparing the ``slice'' and ``product'' distributions, we complete the proof of the first part of the theorem and reduce the second part to a certain polynomial anticoncentration inequality (\cref{better-than-3/4}). This last ingredient almost follows from a result of Fox, Kwan, and Sauermann~\cite{fox-kwan-sauermann-21}, but requires an additional careful argument.

\cref{better-than-1/e} follows from \cref{bulk,far-from-multiple,close-to-multiple} unless $\ell = O(1)$. This last remaining case turns out to be the most challenging. Our proof in this case is computer-assisted: in \cref{reduction-to-finite}, we reduce the problem to a finite computation. While we did not make a serious effort to optimise the constant $0.33$ in the statement of \cref{better-than-1/e}, our proof does provide a computational framework for doing so (see \cref{computational_data}).

\subsection*{Organisation of the paper} In \cref{sec:slice}, we review several known results about polynomials on a slice of the Boolean hypercube. With these tools at hand, in \cref{sec:far-from-multiple}, we give a short proof of \cref{far-from-multiple}. In \cref{sec:close-to-multiple}, we focus on the case when $\ell$ is close to a nonzero multiple of $k$, and present a proof of \cref{close-to-multiple} assuming a certain technical proposition (its proof appears in \cref{sec:proof-better-than-3/4}). In \cref{sec:close-to-zero}, we complete the proofs of \cref{better-than-1/e,2/e^2} by dealing with the case when $\ell$ is close to zero. In \cref{sec:hypergraphs} we turn our attention to the hypergraph setting and prove \cref{hypergraphs-weak}.

Finally, in \cref{sec:concluding} we discuss some possible directions for further research, including a conjecture about set systems of unbounded uniformity.

\subsection*{Notation} 
We use standard asymptotic notation throughout. For functions $f=f(n)$ and $g=g(n)$, we write $f=O(g)$ or $f \lesssim g$ to mean that there is a constant $C$ such that $|f| \le C|g|, f=\Omega(g)$ or $f \gtrsim g$ to mean that there is a constant $c>0$ such that $f(n) \ge c|g(n)|$ for sufficiently large $n$, and $f=o(g)$ to mean that $f / g \rightarrow 0$ as $n \rightarrow \infty$. Subscripts on asymptotic notation indicate quantities that should be treated as constants.

We also use standard graph theory notation. In particular, $V(G)$ and $E(G)$ denote the vertex and edge sets of a graph $G$, respectively, and $e(G) = |E(G)|$. For a vertex $v \in V(G)$, we write $N_G(v)$ for its neighbourhood in $G$ and $\deg_G(v) = |N_G(v)|$ for its degree. For a set of vertices $U \subseteq V(G)$, we write $G[U]$ for the subgraph of $G$ induced by $U$. For two disjoint sets of vertices $U_1, U_2 \subseteq V(G)$, we write $G[U_1, U_2]$ for the bipartite subgraph of $G$ with parts $U_1$ and $U_2$ containing all edges of $G$ between $U_1$ and $U_2$. 

For a positive integer $n$, we write $[n]=\{1, \ldots, n\}$. All logarithms in this paper are to base $e$. All polynomials in this paper have real coefficients. We sometimes omit floor and ceiling symbols and assume large numbers are integers when divisibility considerations are not important.

\section{Polynomials on a slice of the Boolean hypercube}
\label{sec:slice}

\begin{definition} \label{slice}
    Let $\Slice(n, k)$ denote the subset of $\{0, 1\}^n$ with exactly $k$ entries equal to $1$. We write $\vec{\sigma} \sim \Slice(n, k)$ to denote a uniformly random element of $\Slice(n, k)$.
\end{definition}

Many of the classical results for polynomials of Bernoulli random variables can be transferred, with appropriate modifications, to polynomials on the slice. In this section, we collect several such adapted results that will be used throughout the paper.
First, we need the following Azuma--Hoeffding-type concentration inequality, due to Kwan, Sudakov, and Tran~\cite{kwan-sudakov-tran-19}.

\begin{proposition}[{Kwan--Sudakov--Tran~\cite[Lemma 2.1]{kwan-sudakov-tran-19}}] \label{concentration-on-the-slice}
    Consider a function $f: \{0, 1\}^n \to \RR$ such that for every $i \in [n]$ and $x_1, \ldots, x_n \in \{0, 1\}$ 
    \[
    |f(x_1, \ldots, x_{i-1}, 0, x_{i+1}, \ldots, x_n) - f(x_1, \ldots, x_{i-1}, 1, x_{i+1}, \ldots, x_n)| \le a_i.
    \]
    Let $k \le n$ and $\vec{\sigma} \sim \Slice(n, k)$. Then, for any $t > 0$, 
    \[
    \mathbb P\Big[{|f(\vec{\sigma}) - \expected{f(\vec{\sigma})}| \ge t}\Big] \le 2\exp\left(-\frac{t^2}{8 \sum_{i=1}^n a_i^2}\right).
    \] 
\end{proposition}

Erd\H{o}s~\cite{erdos-45}, sharpening a result of Littlewood and Offord~\cite{littlewood-offord-43}, proved that for arbitrary nonzero real numbers $a_1, \ldots, a_n$ and independent Rademacher random variables $\xi_1, \ldots, \xi_n$
\begin{equation} \label{eq:ELO}
\sup_{\ell \in \RR}\prob{a_1 \xi_1 + \ldots + a_n \xi_n = \ell} \le \binom{n}{\lfloor n/2 \rfloor} \cdot 2^{-n} \lesssim \frac{1}{\sqrt{n}}.
\end{equation}
The next proposition, essentially due to Litvak, Lytova, Tikhomirov, Tomczak-Jaegermann, and Youssef~\cite{LLTTY-17}, provides an analogue of this result for linear polynomials on the slice.
\begin{proposition} \label{linear-on-the-slice}
    Let $R, n, k, m \in \NN$ satisfy $2k \le n \le Rk$ and $m \le k$. Consider a linear polynomial $f(x_1, \ldots, x_n) = a_1 x_1 + \ldots + a_n x_n$, and let $\vec{\sigma} \sim \Slice(n, k)$. Then either there exists a set $I \subseteq [n]$, $|I| \ge n-m$ such that all $a_i$ for $i \in I$ are equal, or
    \[
    \sup_{\ell \in \RR} \prob{f(\vec{\sigma}) = \ell} \lesssim_R \frac{1}{\sqrt{m}}.
    \]
\end{proposition}
\begin{proof} 
    \cite[Proposition 4.10]{LLTTY-17} states that for a polynomial $f_0(x_1, \ldots, x_{2k}) = \sum_{i=1}^{2k} b_i x_i$ and $\vec{\sigma}_0 \sim \Slice(2k, k)$, if there exists a set $I_0 \subseteq [2k]$ such that $|I_0| \le k$ and $b_i \neq b_j$ for every $i \in I_0, j \notin I_0$, then 
    \begin{equation} \label{eq:LLTTY}
    \sup_{\ell \in \RR} \prob{f_0(\vec{\sigma}_0) = \ell} \lesssim \frac{1}{\sqrt{|I_0|}}.
    \end{equation}

    Let $I$ be the largest subset of $[n]$ such that $|I| \le n/2$ and 
    \begin{enumerate}
        \item[$(\diamond)$] $a_i \neq a_j$ for every $i \in I, j \notin I$.
    \end{enumerate}
    First, suppose that $|I| < m/2$. Let $x$ be an arbitrary element of $[n] \setminus I$, and let $J = \{j \in [n] : a_j = a_x\}$. Note that $J$ is disjoint from $I$, and that both $I \cup J$ and $[n] \setminus (I \cup J)$ satisfy the property $(\diamond)$ as well. Therefore, we have $|I \cup J| > n/2$. Thus, $|[n] \setminus (I \cup J)| < n/2$, and we conclude (by the choice of $I$) that $|[n] \setminus (I \cup J)| < m/2$. But then $|J| = n - |I| - |[n] \setminus (I \cup J)| > n - m$, and all $a_j$ for $j \in J$ are equal.

    So, we may assume that $|I| \ge m/2$. Note that a uniformly random element of $\Slice(n, k)$ has the same distribution as a uniformly random element of $\Slice(2k, k)$ where the ground set is a uniformly random subset $B$ of $[n]$ of size $2k$. By \cref{concentration-on-the-slice} applied to the polynomial $\sum_{i \in I} x_i$ on $\Slice(n, 2k)$, we conclude that $|B \cap I| \ge (k/n)|I| \ge m/(2R)$ with probability at least $1 - 2\exp(-\Omega_R(m))$ over the choice of $B$. In this case, application of \cref{eq:LLTTY} to the polynomial $\sum_{i \in B} a_i x_i$ and the set $B \cap I$ gives the desired bound of $O_R(1/\sqrt{m})$. 
\end{proof}

To state an extension of the Erd\H{o}s--Littlewood--Offord theorem to polynomials of higher degree, we need the following definitions.

\begin{definition} \label{def:nu_r}
Let $f \in \RR[x_1, \ldots, x_n]$ be an $n$-variable polynomial of degree $r$. Then $\nu_r(f)$ is the maximum integer $m$ such that there exist $m$ disjoint subsets $I_1, \ldots, I_m \subseteq [n]$, $|I_j| = r$ such that for every $j$ the coefficient of $\vec{x}^{I_j} = \prod_{i \in I_j} x_i$ in $f$ is nonzero. 
\end{definition}
\begin{definition} \label{def:LO}
For each $r \in \NN$, we define
\[
\LO_r(m) = \sup_{f, n} \prob{f(\xi_1, \ldots, \xi_n) = 0}, 
\]
where $\xi_1, \ldots, \xi_n$ are independent Rademacher random variables, and the supremum is over all $n \in \NN$ and all $n$-variable polynomials $f$ of degree $r$ with $\nu_r(f) \ge m$.
\end{definition}

In these terms, the Erd\H{o}s--Littlewood--Offord theorem translates to $\LO_1(m) \lesssim 1/\sqrt{m}$. Recently, Kwan and Sauermann~\cite{kwan-sauermann-23} obtained the same optimal bound in the quadratic case: $\LO_2(m) \lesssim 1/\sqrt{m}$. It is widely believed that $\LO_r(m) \lesssim_r 1/\sqrt{m}$ holds true for all $r \ge 1$, but the best known bound (by Meka, Nguyen, and Vu~\cite{meka-nguyen-vu-16} via a result of Kane~\cite{kane-14}) has an extra factor of $(\log m)^{O_r(1)}$.

The following theorem (due to Jain, Kwan, Mubayi, and Tran~\cite{jain-kwan-mubayi-tran-25}) allows one to transfer these results to sufficiently ``sparse'' polynomials on the slice with small non-negative integer coefficients.

\begin{theorem}[{Jain--Kwan--Mubayi--Tran~\cite[Lemma 5.1]{jain-kwan-mubayi-tran-25}}] \label{large-matching}
    For $r, q \in \NN$, there exists $\delta = \delta(r, q) > 0$ such that the following holds. Consider $R, k, m, n \in \NN$ such that $2k \le n \le Rk$. Let $f$ be an $n$-variable multilinear polynomial of degree at most $r$ with coefficients in $\{0, \ldots, q\}$. Suppose that $\nu_r(f) \ge m$ and $f$ has at most $\delta n^r$ nonzero degree-$r$ terms. Then, for $\vec{\sigma} \sim \Slice(n, k)$,
    \[
    \sup_{\ell \in \RR} \prob{f(\vec{\sigma}) = \ell} \lesssim_{r, q, R} \max_{r' \le r} \LO_{r'}(\Omega_{r, q, R}(m)).
    \]
\end{theorem}

For \cref{hypergraphs-weak} (on hypergraph edge-inducibilities), we will also need some minor variants of \cref{large-matching}, whose statements we defer to \cref{sec:hypergraphs}. 

Combining \cref{large-matching} with a concentration inequality (\cref{concentration-on-the-slice}) and with the optimal bounds for the linear and quadratic Littlewood--Offord problems, we obtain the following corollary.

\begin{corollary} \label{large-matching-corollary}
    There exists an absolute constant $\delta_0 > 0$ such that the following holds. Consider $R, k, n, m \in \NN$ such that $2k \le n \le Rk$, and let $f$ be an $n$-variable multilinear quadratic polynomial with coefficients in $\{0, 1\}$, such that $\nu_2(f) \ge m$. Then, for $\vec{\sigma} \sim \Slice(n, k)$ and every $\ell \le \delta_0 k^2$,
    \[
    \prob{f(\vec{\sigma}) = \ell} \lesssim_{R} \frac{1}{\sqrt{m}}.
    \]
\end{corollary}
\begin{proof}
    Let $\delta > 0$ be such that \cref{large-matching} holds with $r = 2$ and $q = 1$. If $f$ has at most $\delta n^2$ nonzero quadratic terms, then we have
    \[
    \sup_{\ell \in \RR}\prob{f(\vec{\sigma}) = \ell} \le \max_{r' = 1, 2} \LO_{r'}(\Omega_R(m))\lesssim_{R} \frac{1}{\sqrt{m}}.
    \]
    So, we can assume that $f$ has at least $\delta n^2$ nonzero quadratic terms. Set $\delta_0 = \delta/4$. Then
    \[
    \expected{f(\vec{\sigma})} \ge \frac{k(k-1)}{n(n-1)} \cdot \delta n^2 \ge 2 \delta_0 k^2.
    \]
    Note that the polynomial $f$ satisfies the assumptions of \cref{concentration-on-the-slice} with $a_i = n+1$ for every $i \in [n]$. Therefore, we conclude that for every $\ell \le \delta_0 k^2$,
    \[
    \prob{f(\vec{\sigma}) = \ell} \le \prob{f(\vec{\sigma}) \le \delta_0 k^2} \le 2\exp\left(-\frac{(\delta_0 k^2)^2}{8n(n+1)^2}\right) \le \exp(-\Omega_R(k)).
    \]
    Since $k \ge n/R \ge m/R$, this bound is certainly $O_R(1/\sqrt{m})$.
\end{proof}

\section{$\ell$ is far from a multiple of $k$}
\label{sec:far-from-multiple}

\begin{proof}[\textbf{Proof of \cref{far-from-multiple}}]
    We may assume that $k$ is sufficiently large. Recall that we have $\ell = ak + \ell_0$ for some $1 \le \ell_0 \le k-1$ and $0 \le a \le \sqrt{k}$, and set $m = \frac{1}{3}\min(\sqrt{\ell_0}, (k-\ell_0)/(a+1))$. In particular, we have $2m^2 < \ell_0$ and $2(a+1)m < k-\ell_0$.

    Also recall that, by definition, $\ind(k, \ell)$ is the limit of a non-increasing sequence $N(n, k, \ell) / \binom{n}{k}$, where $N(n, k, \ell)$ is the maximum number of $k$-vertex subsets of an $n$-vertex graph that induce exactly $\ell$ edges. Thus, it is sufficient to prove that $N(n, k, \ell) / \binom{n}{k} \lesssim 1/\sqrt{m}$ for $n = 3k$.

    Consider a graph $G$ on $n = 3k$ vertices. Let $X$ be a uniformly random $k$-subset of $V(G)$. Then $e(G[X])$ can be interpreted as $f(\vec{\sigma})$ for $\vec{\sigma} \sim \Slice(n, k)$ and some homogeneous multilinear quadratic polynomial $f$ in $n$ variables with coefficients in $\{0, 1\}$. We need to prove that 
    \[
    \prob{e(G[X]) = \ell} = \prob{f(\vec{\sigma}) = \ell} \lesssim 1/\sqrt{m},
    \]
    where $\ell \le (\sqrt{k} + 1) k$. If $\nu_2(f) \ge m/2$, then this bound follows from \cref{large-matching-corollary}. 
    
    Thus, we may assume that $\nu_2(f) < m/2$. In this case, there exists a set $U \subseteq [n]$, $|U| \le m$ such that every quadratic term of $f$ contains at least one variable $x_{u}$ for $u \in U$. For the rest of the proof, we condition on an arbitrary outcome $U'$ of $U \cap X$. In terms of $f$, this corresponds to substituting $x_u = 0$ for every $u \in U \setminus U'$, and $x_u = 1$ for every $u \in U'$. After substitution, the polynomial becomes of degree at most one in the remaining variables (denote it by $f_{U'}$). Then we need to estimate the probability $\prob{f_{U'}(\vec{\sigma}_0) = \ell}$ for $\vec{\sigma}_0 \sim \Slice(n - |U|, k -|U'|)$. 

    Note that the constant term of $f_{U'}$ is at most $|U'|^2 \le m^2$, and that the coefficients of its linear terms lie in $\{0, 1, \ldots, m\}$. If no $(n-|U|)-m$ of these linear coefficients are equal, then by the anticoncentration result for linear polynomials on the slice (\cref{linear-on-the-slice}), we obtain the desired bound $\prob{f_{U'}(\vec{\sigma}_0) = \ell} \lesssim 1/\sqrt{m}$.

    So, we may assume that all but $m$ linear coefficients of $f_{U'}$ are equal to some integer $a_0 \in \{0, 1, \ldots, m\}$. In this case, we will prove that $\prob{f_{U'}(\vec{\sigma}_0) = \ell} = 0$. Indeed, note that $f_{U'}(\vec{\sigma}_0)$ then satisfies the following inequalities (with probability $1$):
    \[
    a_0(k - 2m) \le a_0(k - |U'| - m) \le f_{U'}(\vec{\sigma}_0) \le a_0 k + |U'| \cdot m + m^2 \le a_0k + 2m^2.
    \]
    Since $2m^2 < \ell_0$, for $a_0 \le a$ we have  
    \[
    f_{U'}(\vec{\sigma}_0) \le a_0 k + 2m^2 < a k + \ell_0 = \ell.
    \]
    Since $2 (a+1) m < k-\ell_0$, for $a_0 \ge a+1$ we have
    \[
    f_{U'}(\vec{\sigma}_0) \ge a_0(k-2m) > (a+1)k - (k-\ell_0) = \ell. \qedhere
    \]    
\end{proof}

\section{$\ell$ is close to a nonzero multiple of $k$}
\label{sec:close-to-multiple}

In this section, we focus on the case when $\ell$ is close to $a(k-a)$ for some fixed integer $a \ge 1$; our goal is to prove \cref{close-to-multiple}.

For $p \in [0, 1]$ and a function $F: \{0, 1\}^n \to \RR$, we write $F(\vec{\xi}(p))$ as a shorthand for $F(\xi_1, \ldots, \xi_n)$, where $\xi_1, \ldots, \xi_n$ are independent $\Ber(p)$ random variables.
For discrete probability distributions $\mu, \nu$ on the real line, recall that the \emph{total variation distance} $\mathrm{d_{TV}}(\mu, \nu)$ is the supremum of $|\mu(S) - \nu(S)|$ over all $S \subseteq \RR$. 

In \cref{sec:slice}, we collected several anticoncentration inequalities for polynomials on a slice of the Boolean hypercube, mirroring classical results for polynomials of independent Bernoulli variables. In the special case when the polynomial depends only on a few variables, the following result (essentially due to Ehm~\cite{ehm-91}, as observed in \cite{jain-kwan-mubayi-tran-25}) provides a direct way of comparing these two settings.

\begin{theorem}[{\cite[Theorem 7.1]{jain-kwan-mubayi-tran-25}}] \label{product-slice}
    Let $n, k, s \in \NN$ satisfy $k \le n/2$ and $s \le n$. Consider a function $F: \{0, 1\}^n \to \RR$ which depends only on $x_1, \ldots, x_s$. Then, for $\vec{\sigma} \sim \Slice(n, k)$ 
    \[
    \mathrm{d_{TV}}\left(F(\vec{\sigma}), F(\vec{\xi}(k/n))\right) \le \frac{\max(s, 2n/k) - 1}{n-1} \le \max(s/n, 3/k).
    \]
    In particular,
    \[
    \left| \prob{F(\vec{\sigma}) = 0} - \prob{F(\vec{\xi}(k/n)) = 0} \right| \le \max(s/n, 3/k).
    \]
\end{theorem}

Fox, Kwan, and Sauermann~\cite{fox-kwan-sauermann-21} observed that the following proposition is a consequence of Alon's Combinatorial Nullstellensatz~\cite{alon-99}.

\begin{proposition}[{Fox--Kwan--Sauermann~\cite[Proposition 1.9]{fox-kwan-sauermann-21}}] \label{FKS-nullstellensatz}
    Let $f$ be a polynomial of degree at most $r$ with a nonzero constant term. Then, for each $p \in [0, 1/2]$,
    \[
    \prob{f(\vec{\xi}(p)) = 0} \le 1 - 2^{-r}.
    \]
\end{proposition}
While this is sharp for $p = 1/2$, it seems plausible that one can obtain a better bound for small $p$. In this direction, Fox, Kwan, and Sauermann~\cite[Theorem 1.8]{fox-kwan-sauermann-21} proved that for any $\ell \neq 0$ and a polynomial $f$ with \emph{non-negative} coefficients and no constant term, one has $\prob{f(\vec{\xi}(p)) = \ell} \le 1/e + o(1)$ as $p \to 0$.
In the case $r = 2$, \cref{FKS-nullstellensatz} gives the bound of $3/4$.
We obtain the following (marginal) improvement over this bound for a certain class of quadratic polynomials.

\begin{proposition} \label{better-than-3/4}
    Let $p > 0$ be sufficiently small. Consider a quadratic polynomial $f$ with a nonzero constant term such that its degree-$2$ coefficients lie in $\{0, 1\}$. Then, 
    \[
    \prob{f(\vec{\xi}(p)) = 0} < 0.725.
    \]
\end{proposition}

The proof of \cref{better-than-3/4} involves a somewhat tedious case analysis and is postponed until \cref{sec:proof-better-than-3/4}.
Our motivation to prove it comes from the numerical inequalities
\[
0.725 \cdot \frac{1}{e} < \frac{2}{e^2} < \frac{3}{4} \cdot \frac{1}{e}.
\]
Therefore, an attempt to prove \cref{close-to-multiple}(2) using \cref{FKS-nullstellensatz} instead of \cref{better-than-3/4} would yield a slightly worse bound in the case $a = 1$. This would still be sufficient to deduce \cref{better-than-1/e} but not \cref{2/e^2}, since the latter asks for the optimal bound of $2/e^2 + o(1)$.
We also remark that the bound in \cref{better-than-3/4} is likely not optimal: the best lower bound we know is $2/e - 1/e^2 + o(1) \approx 0.6$ (attained by the polynomial $f(x_1, \ldots, x_{2n}) = (x_1 + \ldots + x_n - 1)(x_{n+1} + \ldots + x_{2n} - 1)$ with $p = 1/n$).

Next, we need the following classical result about antichains in the Boolean hypercube, discovered independently by Bollob\'as~\cite{bollobas}, Lubell~\cite{lubell}, Yamamoto~\cite{yamamoto}, and Meshalkin~\cite{meshalkin}.

\begin{proposition}[BLYM inequality] \label{BLYM}
    Let $\mc{A} \subseteq 2^{[N]}$ be an antichain (that is, there are no $A_1, A_2 \in \mc{A}$ such that $A_1 \subsetneq A_2$). Then, 
    \[
    \sum_{A \in \mc{A}} \binom{N}{|A|}^{-1} \le 1.
    \]
\end{proposition}

Also, we use the following statement about approximating the binomial distribution with the Poisson distribution, which follows (for example) from \cite[Theorem 1]{barbour-hall-1984}.

\begin{proposition} \label{Poisson-formula} 
    For each $p \in (0, 1)$ and $N \in \NN$, we have $\mathrm{d_{TV}}(\Bin(N, p), \Poi(pN)) \le p$. In particular, for every $0 \le M \le N$, 
    \[
    \Big|\prob{\Bin(N, p) = M} - \prob{\Poi(pN) = M}\Big| = \left|\binom{N}{M} p^M (1-p)^{N-M} - \frac{(pN)^M}{e^{pN} M!}\right| \le p.
    \]
\end{proposition}

Next, we combine \cref{BLYM,Poisson-formula} to derive the following convenient lemma.

\begin{lemma} \label{antichain-expectation}
    Let $\mc{A}$ be an antichain of subsets of some finite set $B$, and let $\varphi: 2^{B} \to [0, 1]$ be a function such that $\varphi(A) = 0$ for every $A \notin \mc{A}$. Then, for each $p \in (0, 1)$ and a $p$-random subset $B_p$ of $B$ (including each element of $B$ with probability $p$ independently), 
    \[
    \expected{\varphi(B_p)} \le \max_{A \in \mc{A}} \left(\frac{|A|^{|A|}}{e^{|A|} |A|!} \varphi(A)\right) + p.
    \] 
\end{lemma}
\begin{proof}
    \cref{Poisson-formula}, combined with a standard calculation, implies that for $0 \le M \le N$, 
    \begin{equation} \label{eq:use-Poisson}
    \binom{N}{M} p^M (1-p)^{N-M} \le \frac{(pN)^M}{e^{pN} M!} + p \le \frac{M^M}{e^M M!} + p.        
    \end{equation}
    Therefore,
    \begin{align*}
    \expected{\varphi(B_p)} = \sum_{A \in \mc{A}} p^{|A|} (1-p)^{|B| - |A|}\varphi(A) &\le \max_{A \in \mc{A}} \left(\binom{|B|}{|A|} \cdot p^{|A|} (1-p)^{|B| - |A|} \varphi(A)\right) \\
    &\le \max_{A \in \mc{A}} \left(\Big(\frac{|A|^{|A|}}{e^{|A|} |A|!} + p\Big) \varphi(A)\right) \le \max_{A \in \mc{A}} \left(\frac{|A|^{|A|}}{e^{|A|} |A|!}\varphi(A)\right) + p.
    \end{align*}
    Here in the first line we used that $\varphi$ is supported on $\mc{A}$, and that $\sum\limits_{A \in \mc{A}} \binom{|B|}{|A|}^{-1} \le 1$ by \cref{BLYM}. In the second line, the first inequality is by \cref{eq:use-Poisson}, and the second inequality follows from the fact that $0 \le \varphi(A) \le 1$.
\end{proof}

\begin{proof}[\textbf{Proof of \cref{close-to-multiple}}]
    We write $(\a_1, \ldots, \a_M) \gg (\beta_1, \ldots, \beta_N)$ as a shorthand for ``each of $\a_1, \ldots, \a_M$ is sufficiently large in terms of $\beta_1, \ldots, \beta_N$''. Our goal is to prove the upper bound on $\ind(k, \ell)$ for $k \gg (a, C, 1/\eps)$.  For convenience, we introduce three intermediate parameters $m, t$, and $R$ satisfying 
    \[
    R \gg 1/\eps, \qquad k \gg (m, t) \gg (R, a, C, 1/\eps).
    \]

    Recall that, by definition, $\ind(k, \ell)$ is the limit of a non-increasing sequence $N(n, k, \ell) / \binom{n}{k}$. Thus, it is sufficient to prove the upper bound on $N(n, k, \ell) / \binom{n}{k}$ for $n = Rk$.     
    Let $G$ be a graph on the vertex set $V$ of size $n = Rk$, and let $X$ be a uniformly random $k$-subset of $V$. Then we need to estimate $\prob{e(G[X]) = \ell}$. We may assume that $G$ has at least $\ell$ edges (otherwise, this probability is zero).

    Since $m \gg (R, a)$, if $G$ has a matching of size $m/2$ then the desired bound follows from \cref{large-matching-corollary}. Therefore, we can assume that the maximum matching in $G$ has size less than $m/2$. Let $U$ be the vertex set of this matching. Then $|U| \le m$, and $V \setminus U$ is an independent set in $G$. 
    
    For each $u \in U$, denote $d(u) = |N_G(u) \cap (V \setminus U)| \le n$. Next, we find a subset $W \subseteq U$ consisting of its ``highest degree'' vertices. 

    \begin{claim} \label{large-degrees}
        There exists a non-empty subset $W \subseteq U$ such that for each $w \in W$, we have $d(w) \gtrsim_{m, t, C} k$ and $d(w) \ge 2t \cdot e(G[V \setminus W])$.
    \end{claim}
    \begin{claimproof}
        Let $u_1, \ldots, u_{|U|}$ be the vertices of $U$ ordered by non-increasing values of $d(u_i)$: that is, $d(u_i) \ge d(u_{i+1})$ for every $i \ge 1$. First we note that $d(u_1) \gtrsim_{m, C} k$. Indeed, recalling that $V \setminus U$ is an independent set and that $e(G) \ge \ell \ge k - C$, we obtain 
        \[
        m \cdot d(u_1) \ge e(G[U, V \setminus U]) \ge \ell - |U|^2 \ge k - C - m^2.
        \]

        Let $j$ be the smallest index $i$ such that $d(u_i) > 4tm \cdot d(u_{i+1})$ (or $j = |U|$ if no such index $i$ exists), and set $W = \{u_1, \ldots, u_j\}$. Then, for every $1 \le j_0 \le j$, 
        \[
        d(u_{j_0}) \ge d(u_j) \ge d(u_1) / (4tm)^m \gtrsim_{m, t, C} k.
        \]
        To check the second condition, we first observe that $e(G[U \setminus W, V \setminus U]) \le d(u_j)/(4t)$. Indeed, if $W = U$ then this is trivially true. Otherwise,
        \[
        e(G[U \setminus W, V \setminus U]) \le \sum_{i = j+1}^{|U|} d(u_i) \le m \cdot d(u_{j+1}) \le d(u_j)/(4t)
        \]
        where the last inequality is by the definition of $j$. Therefore,
        \[
        e(G[V \setminus W]) \le e(G[U \setminus W, V \setminus U]) + |U|^2 \le d(u_j) / (4t) + m^2 \le d(u_j) / (2t),
        \]
        where in the last inequality we used that $d(u_j) \gtrsim_{m, t, C} k$ and $k \gg (m, t, C)$.
    \end{claimproof}

    Let $d = \min\limits_{w \in W} d(w) \ge 2t \cdot e(G[V \setminus W])$. Call a subset $W' \subseteq W$ \emph{good} if 
    \[
    \Big| \frac{1}{R} \sum_{w \in W'} d(w) - \ell \Big| \le \frac{d}{t}.
    \]
    Note that, since $t \gg R$, the good subsets of $W$ form an antichain. Furthermore, as $\ell \ge ak - C$, for every good subset $W'$ we have
    \[
    ak - C \le \ell \le \frac{1}{R} \sum_{w \in W'} d(w) + \frac{d}{t} \le \frac{n \cdot |W'|}{R} + \frac{n}{t} = k \cdot \left(|W'|+ \frac{R}{t}\right).
    \]
    Since $t \gg R$, and $k \gg (a, C)$, this implies that $|W'| \ge a$.    Next, we check that if $W \cap X$ is not good, then we are unlikely to have $e(G[X]) = \ell$.

    \begin{claim} \label{bad-W'}
        $\prob{e(G[X]) = \ell \text{ and } W \cap X \text{ is not good}} \le \exp(-\Omega_{m, t, R, C}(k))$.
    \end{claim}

    \begin{claimproof}
        It suffices to prove that for every $W' \subseteq W$ which is not good, 
        \[
        \prob{e(G[X]) = \ell \mid W \cap X = W'} \le \exp(-\Omega_{m, t, R, C}(k)).
        \]
        Suppose that $W \cap X = W'$ and $e(G[X]) = \ell$. Then,
        \[
        \ell = e(G[X]) = e(G[X \cap U]) + e(G[X \cap (U \setminus W), X \cap (V \setminus U)]) + \sum_{w \in W'} |N_G(w) \cap (X \cap (V \setminus U))|.
        \]
        The first term is at most $|U|^2 \le m^2$. The second term is at most $d/(2t)$ by \cref{large-degrees}. Therefore, as $W'$ is not good,
        \begin{equation} \label{eq:bad-W'}
        \left|\sum_{w \in W'} |N_G(w) \cap (X \cap (V \setminus U))| - \frac{1}{R} \sum_{w \in W'} d(w) \right| \ge \frac{d}{t} - \frac{d}{2t} - m^2 = \frac{d}{2t} - m^2.
        \end{equation}
        On the other hand, $\sum\limits_{w \in W'} |N_G(w) \cap (X \cap (V \setminus U))|$ may be interpreted as a linear polynomial on $\Slice(n-|W|, k-|W'|)$, and its expected value satisfies
        \[
        \left|\,\mathbb E\left[{\sum_{w \in W'} |N_G(w) \cap (X \cap (V \setminus U))|}\right] - \frac{1}{R}\sum_{w \in W'} d(w)\right| = \left| \frac{k - |W'|}{n - |W|} - \frac{1}{R}\right| \cdot \sum_{w \in W'} d(w) \lesssim \frac{m}{n} \cdot m n = m^2.
        \]
        Since $d \gtrsim_{m, t, C} k$, and $k \gg (m, t, C)$, we note that if \cref{eq:bad-W'} holds, then this polynomial takes a value at least $d/(4t)$ away from its expectation. Applying \cref{concentration-on-the-slice} (our concentration inequality on the slice) with $a_v = |W'| \le m$, we conclude that this happens with probability at most
        \[
        2\exp\left(\frac{-(d/(4t))^2}{8 m^2 n}\right) \le \exp(-\Omega_{m, t, R, C}(k)),
        \]
        where in the last inequality we used that $d \gtrsim_{m, t, C} k$ and $n = Rk$.
    \end{claimproof}    
    Define a function $\varphi: 2^W \to [0, 1]$ as follows: if $W' \subseteq W$ is good, then let 
    \[
    \varphi(W') = \prob{e(G[X]) = \ell \mid W \cap X = W'};
    \] 
    otherwise, set $\varphi(W') = 0$.
    Since $k \gg (m, t, R, C, 1/\eps)$, \cref{bad-W'} implies that
    \begin{equation} \label{eq:sum-good-W'}
    \prob{e(G[X]) = \ell} \le \prob{e(G[X]) = \ell \text{ and } W \cap X \text{ is good}} + \eps/4 = \expected{\varphi(W \cap X)} + \eps/4.
    \end{equation}
    Set $p = k/n = 1/R$ and let $V_p$ denote a $p$-random subset of $V$. Consider the function $F:\{0,1\}^V\to [0,1]$ which sends the characteristic vector $\mathds 1_A\in \{0,1\}^V$ of a set $A\subseteq V$ to the value $\varphi(W \cap A)$, and note that $F$ depends on at most $|W| \le m$ coordinates. Applying \cref{product-slice} to this function $F$, we obtain
    \begin{align*}
    \Big|\expected{\varphi(W \cap X)} - \expected{\varphi(W \cap V_p)} \Big| &\le \int_{0}^1 \Big|\prob{\varphi(W \cap X) \ge t} - \prob{\varphi(W \cap V_p) \ge t}\Big| \,dt \\
    &\le \mathrm{d_{TV}}(\varphi(W \cap X), \varphi(W \cap V_p)) \le \max\left(\frac{|W|}{n}, \frac{3}{k}\right) \le \frac{m+3}{k}.
    \end{align*}
    But $W \cap V_p$ is just a $p$-random subset $W_p$ of $W$. Then, since $k \gg (m, 1/\eps)$, we can bound the expression from \cref{eq:sum-good-W'} as
    \[
    \prob{e(G[X]) = \ell} \le \expected{\varphi(W_p)} + (m+3)/k + \eps/4 \le \expected{\varphi(W_p)} + \eps / 2.
    \]
    Recall that $\varphi$ is supported on the antichain of good subsets of $W$, and that $p = 1/R \le \eps / 4$ (because $R \gg 1/\eps$). Therefore, we can bound the expected value $\expected{\varphi(W_p)}$ using \cref{antichain-expectation}, and obtain that
    \begin{equation} \label{eq:apply-antichain-expectation}
    \prob{e(G[X]) = \ell} \le \expected{\varphi(W_p)} + \eps/2 \le \max_{W' \subseteq W \text{ is good}} \left(\frac{|W'|^{|W'|}}{e^{|W'|} |W'|!} \varphi(W')\right) + 3\eps / 4.
    \end{equation}
    Recall that every good subset has size at least $a$, and that $M^M/(e^M M!)$ is a decreasing function of $M \in \NN$.
    Now we can complete the proof of the first part of the theorem. Indeed, using the crude bound $\varphi(W') \le 1$, we see that the maximum in the right hand side of \cref{eq:apply-antichain-expectation} is at most $a^a/(e^a a!)$, and hence
    \[
    \prob{e(G[X]) = \ell} \le \frac{a^a}{e^a a!} + 3\eps/4 \le \frac{a^a}{e^a a!} + \eps.
    \]
    To prove the second part, we need a non-trivial bound on $\varphi(W')$ for good subsets of size exactly $a$. If $W'$ is a good subset of size $a$, then 
    \begin{equation} \label{eq:sumd}
    \sum_{w \in W'} d(w) \ge R(\ell - d/t) \ge R(a(k-a) - C - d/t) \ge an - O_{R, a, C}(n/t).
    \end{equation}
    Therefore, each vertex of $W'$ is adjacent to all but at most $O_{R, a, C}(n/t)$ other vertices of $G$.
    Conditioning on the event $W \cap X = W'$, we interpret $e(G[X])$ as a polynomial of a random variable $\vec{\sigma} \sim \Slice(n - |W|, k - |W'|)$:   
    \[
    e(G[X]) = e(G[W']) + \sum_{\substack{u \in W', v \in V \setminus W \\ uv \in E(G)}} \sigma_{v} + \sum_{uv \in E(G[V \setminus W])} \sigma_{u} \sigma_{v}.
    \]
    As discussed above, the sum in the second term is taken over \emph{almost all} pairs of vertices $u \in W'$ and $v \in V \setminus W$. Thus, we assume that ``by default'' each vertex $u \in W'$ is adjacent to all $k-a$ vertices of $X \cap (V \setminus W)$, and then subtract the corresponding sum over the non-edges. Namely, we can write
    \[
    e(G[X]) = e(G[W']) + a(k-a) + f(\vec{\sigma}),
    \]
    where 
    \[
    f((x_v)_{v \in V \setminus W}) =  - \sum_{\substack{u \in W', v \in V \setminus W \\ uv \notin E(G)}} x_{v} + \sum_{uv \in E(G[V \setminus W])} x_{u} x_{v}.
    \]
    Note that, by \cref{eq:sumd}, the first term involves only $O_{R, a, C}(n/t)$ variables. \cref{large-degrees} implies that $e(G[V \setminus W]) \le n/(2t)$, thus the second term also involves at most $n/t$ variables. Therefore, by \cref{product-slice}, $f$ is essentially a polynomial of Bernoulli random variables. Specifically, let 
    \[
    \ell^* = \ell - e(G[W']) - a(k-a), \qquad p^* = (k-|W'|)/(n - |W|) \lesssim 1/R.
    \]
    Then, \cref{product-slice} implies that
    \begin{equation} \label{eq:slice-to-product}
    \varphi(W') = \prob{e(G[X]) = \ell \mid W \cap X = W'} = \prob{f(\vec{\sigma}) = \ell^*} \le \prob{f(\vec{\xi}(p^*)) = \ell^*} + O_{R, a, C}(\max(1/t, 1/k)). 
    \end{equation}
    Recall that $|W'| = a$, and that, by the assumption in the second part of the theorem, $\ell \notin [a(k-a), a(k-a) + \binom{a}{2}]$. This implies that $\ell^* \neq 0$, and thus the polynomial $f - \ell^*$ satisfies the assumptions of \cref{better-than-3/4}. So, 
    \[
    \prob{f(\vec{\xi}(p^*)) = \ell^*} < 0.725.
    \]
    Since $(k, t) \gg (R, a, C, 1/\eps)$, from \cref{eq:slice-to-product} we can further conclude that
    \[
    \varphi(W') < 0.725 + \eps/4.
    \]
    Substituting this into \cref{eq:apply-antichain-expectation}, we obtain the bound
    \[
    \prob{e(G[X]) = \ell} \le \max\left(0.725 \cdot \frac{a^a}{e^a a!}, \frac{(a+1)^{a+1}}{e^{a+1} (a+1)!}\right) + \eps.
    \]
    A simple calculation shows that the second term in the maximum always dominates, completing the proof.
\end{proof}

\section{Proof of \cref{better-than-3/4}} 
\label{sec:proof-better-than-3/4}

In this section we present the proof of \cref{better-than-3/4} (used in the proof of \cref{close-to-multiple}), which states that $\prob{f(\vec{\xi}(p)) = 0} < 0.725$ for each sufficiently small $p > 0$ and each polynomial $f$ satisfying the following assumption:

$(\star)$ $f$ is a quadratic polynomial with a nonzero constant term whose degree-$2$ coefficients lie in $\{0, 1\}$.

We introduce some notation for the maximum point concentration probability of the binomial distribution: for each $m \in \NN$ and $p \in [0, 1]$, we define
\begin{equation} \label{eq:binmax-definition}
\binmax{m}{p} = \max_{0 \le m_0 \le m} \prob{\Bin(m, p) = m_0}, \qquad 
\binmaxplus{m}{p} = \max_{1 \le m_0 \le m} \prob{\Bin(m, p) = m_0}.
\end{equation}
It is a standard fact that for $p < 1$ the maximum in the definition of $\binmax{m}{p}$ is attained at $\lfloor(m+1)p\rfloor$ (i.e., $\lfloor(m+1)p\rfloor$ is a mode of $\Bin(m,p)$), and thus $\binmaxplus{m}{p} = \binmax{m}{p}$ for every $p \in [\frac{1}{m+1}, 1)$. Also note that, for each $m_0$, we have 
\[
\prob{\Bin(m, p) = m_0} = p \cdot \prob{\Bin(m-1, p) = m_0-1} + (1-p) \cdot \prob{\Bin(m-1, p) = m_0} \le \binmax{m-1}{p},
\]
and hence $\binmax{m}{p}$ is a non-increasing function of $m$.

\begin{proof}[\textbf{Proof of \cref{better-than-3/4}}]
First, we observe that it suffices to prove the statement of the proposition for one specific value of $p > 0$. Indeed, consider $p^* < p$ and a polynomial $g^*$ satisfying assumption $(\star)$. Define a (random) polynomial $g$ obtained from $g^*$ in the following way: for each variable $x_i$ of $g^*$ independently, we substitute $x_i = 0$ with probability $1-(p^*/p)$. Note that $g$ also satisfies assumption $(\star)$, and that the random variables $g(\vec{\xi}(p))$ and $g^*(\vec{\xi}(p^*))$ have the same distribution. Therefore, applying the statement of the proposition for $p$ and $g$, we obtain the same statement for $p^*$ and $g^*$:
\[
\prob{g^*(\vec{\xi}(p^*)) = 0} = \mathbb E_g\Big[{\probs{\vec{\xi}(p)}{g(\vec{\xi}(p)) = 0\,|\,g}}\Big] < 0.725.
\]

So, from now on, we fix $p = 0.388$. Consider a polynomial $f$ in the variables $x_1, \ldots, x_s$ satisfying assumption $(\star)$, with a constant term equal to some $\ell \neq 0$. Since any Bernoulli random variable $\xi$ satisfies $\xi^2 = \xi$, we can also assume that $f$ is multilinear.

Suppose that some linear monomial $x_i$ appears in $f$ with a coefficient not equal to $-\ell$. In this case, the polynomials $f^{(0)}_i$ and $f^{(1)}_i$ obtained from $f$ by the substitutions $x_i = 0$ and $x_i = 1$, respectively, also satisfy assumption $(\star)$. Applying the induction hypothesis (on the number of variables) to each of them, we conclude that
\[
\prob{f(\vec{\xi}(p)) = 0} = (1-p) \cdot \prob{f^{(0)}_i(\vec{\xi}(p)) = 0} + p \cdot \prob{f^{(1)}_i(\vec{\xi}(p)) = 0} < ((1-p) + p) \cdot 0.725 = 0.725.
\]
Therefore, we may assume that every linear monomial of $f$ has a coefficient equal to $-\ell$. 

Define a graph $G$ on the vertex set $[s]$ with an edge $ij$ for each monomial $x_i x_j$ appearing in $f$ with coefficient $1$. We consider several cases depending on the structure of $G$: namely, we deal with the case when $G$ is not a complete multipartite graph using \cref{not-complete-multipartite}, with the case when $G$ is complete multipartite but not complete using \cref{complete-multipartite}, and with the case when $G$ is complete using \cref{two-layers}.

\begin{claim} \label{not-complete-multipartite}
Suppose that there exist vertices $u, w_1, w_2$ such that $w_1 w_2$ is an edge in $G$ while $uw_1$ and $uw_2$ are non-edges in $G$. Then, 
\[
\prob{f(\vec{\xi}(p)) = 0} \le \left(1-\binmaxplus{\deg_G(u)}{p}\right) \cdot (1-p) + \binmaxplus{\deg_G(u)}{p} \cdot (1-p^2).
\]
\end{claim}

\begin{claimproof}
    Denote the linear polynomial $\sum_{v \in N_G(u)} x_v$ by $L_{u}$. Then, $f$ can be written as
    \[
    f = \ell + x_u \cdot \left(- \ell + L_u \right) + g
    \]
    for some polynomial $g$ not involving $x_u$.
    Let $\xi_1, \ldots, \xi_s$ be independent $\Ber(p)$ random variables, and write
    \[
    \vec{\xi}_{\mathrm{in}} = (\xi_v)_{v \in N_G(u)}, \quad \vec{\xi}_{\mathrm{out}} = (\xi_v)_{v \notin N_G(u)}.
    \]
    Then, conditioning on the outcome of $\vec{\xi}_{\mathrm{in}}$, we have  
    \begin{equation} \label{eq:conditioning}
    \prob{f(\vec{\xi}_{\mathrm{in}}, \vec{\xi}_{\mathrm{out}}) = 0} = \mathbb E_{\vec{\xi}_{\mathrm{in}}}\Big[{\probs{\vec{\xi}_{\mathrm{out}}}{\ell + \xi_u(-\ell + L_u(\vec{\xi}_{\mathrm{in}})) + g_{\vec{\xi}_{\mathrm{in}}}(\vec{\xi}_{\mathrm{out}}) = 0 \;\big|\; \vec{\xi}_{\mathrm{in}}}}\Big],
    \end{equation}
    where $g_{\vec{\xi}_{\mathrm{in}}}$ is obtained from $g$ by the substitutions $x_v = \xi_v$ for each $v \in N_G(u)$.

    First, we estimate the inner conditional probability in the case when $L_u(\vec{\xi}_{\mathrm{in}}) \neq \ell$. Recalling that the polynomial $g$ does not depend on $x_u$, we can further condition on the outcomes of all the remaining variables except $\xi_u$ to conclude that, for $\vec{\xi}_{\mathrm{in}}$ satisfying $L_u(\vec{\xi}_{\mathrm{in}}) \neq \ell$,
    \begin{equation} \label{eq:ell'-neq-ell}
    \probs{\vec{\xi}_{\mathrm{out}}}{\ell + \xi_u(-\ell + L_u(\vec{\xi}_{\mathrm{in}})) + g_{\vec{\xi}_{\mathrm{in}}}(\vec{\xi}_{\mathrm{out}}) = 0\;\big|\; \vec{\xi}_{\mathrm{in}}} \le \sup_{c \in \RR} \prob{\xi_u = c} \le 1 - p.
    \end{equation}
    Since $\ell \neq 0$, we can bound the probability that $L_u(\vec{\xi}_{\mathrm{in}}) = \ell$ as
    \begin{equation} \label{eq:use-mlp+}
    \probs{\vec{\xi}_{\mathrm{in}}}{L_u(\vec{\xi}_{\mathrm{in}}) = \ell} \le \binmaxplus{\deg_G(u)}{p}.
    \end{equation}
    To bound the inner conditional probability in \cref{eq:conditioning} in the case when $L_u(\vec{\xi}_{\mathrm{in}}) = \ell$, we recall that $g$ contains a quadratic term $x_{w_1} x_{w_2}$ with $w_1, w_2 \notin N_G(u)$. Therefore, conditioning on the outcomes of all the remaining variables except $\xi_{w_1}$ and $\xi_{w_2}$, we conclude that, for $\vec{\xi}_{\mathrm{in}}$ satisfying $L_u(\vec{\xi}_{\mathrm{in}}) = \ell$,
    \begin{equation}\begin{aligned} \label{eq:ell'-eq-ell}
    \probs{\vec{\xi}_{\mathrm{out}}}{\ell + \xi_u(-\ell + L_u(\vec{\xi}_{\mathrm{in}})) + g_{\vec{\xi}_{\mathrm{in}}}(\vec{\xi}_{\mathrm{out}}) = 0\;\big|\; \vec{\xi}_{\mathrm{in}}} &= \probs{\vec{\xi}_{\mathrm{out}}}{\ell + g_{\vec{\xi}_{\mathrm{in}}}(\vec{\xi}_{\mathrm{out}}) = 0\;\big|\; \vec{\xi}_{\mathrm{in}}} \\
    &\le \sup_{c_0, c_1, c_2 \in \RR} \prob{\xi_{w_1} \xi_{w_2} + c_1 \xi_{w_1} + c_2 \xi_{w_2} + c_0 = 0}.
    \end{aligned}\end{equation}
    By the Combinatorial Nullstellensatz (or a direct verification), the 2-variable polynomial $\xi_{w_1} \xi_{w_2} + c_1 \xi_{w_1} + c_2 \xi_{w_2} + c_0$ cannot be identically zero on $\{0, 1\}^2$, and hence the above supremum of probabilities is at most $1 - p^2$.

    Combining \cref{eq:ell'-neq-ell,eq:use-mlp+,eq:ell'-eq-ell}, we can bound the expression from \cref{eq:conditioning} as
    \begin{align*}
    \prob{f(\vec{\xi}(p)) = 0} &\le \prob{L_u(\vec{\xi}_{\mathrm{in}}) \neq \ell} \cdot (1-p) + \prob{L_u(\vec{\xi}_{\mathrm{in}}) = \ell} \cdot (1-p^2) \\
    &\le \left(1-\binmaxplus{\deg_G(u)}{p}\right) \cdot (1-p) + \binmaxplus{\deg_G(u)}{p} \cdot (1-p^2). \qedhere
    \end{align*}
\end{claimproof}
 
By the discussion after \cref{eq:binmax-definition}, for each $m \ge 2$ we have
\begin{equation} \label{eq:0.475}
\binmaxplus{m}{p} = \binmax{m}{p} \le \binmax{2}{p} = 2p(1-p) < 0.475.
\end{equation}
Since $\binmaxplus{0}{p} = 0$ and $\binmaxplus{1}{p} = p = 0.388$, we conclude that $\binmaxplus{m}{p} < 0.475$ for every $m \ge 0$.

Note that the vertices $u, w_1, w_2$ required in the statement of \cref{not-complete-multipartite} exist if and only if the complement of $G$ is not a union of cliques (equivalently, if $G$ is not a complete multipartite graph).
So, in this case \cref{not-complete-multipartite} yields the desired bound
\[
\prob{f(\vec{\xi}(p)) = 0} \le (1 - 0.475) \cdot (1 - 0.388) + 0.475 \cdot (1 - 0.388^2) < 0.725.
\] 
Therefore, we may assume that $G$ \emph{is} a complete multipartite graph.

\begin{claim} \label{complete-multipartite}
Suppose that $I$ is an independent set of $G$ such that each vertex of $I$ is adjacent to each vertex of $[s] \setminus I$. Then,
\[
\prob{f(\vec{\xi}(p)) = 0} \le 1 - \Big(1 - \binmax{|I|}{p}\Big) \Big(1 - \binmaxplus{s-|I|}{p}\Big).
\]
\end{claim}
\begin{claimproof}
The proof is similar to that of \cref{not-complete-multipartite} (and is, in fact, simpler). Denote $J = [s] \setminus I$, and let $L_I = \sum_{v \in I} x_v$, $L_J = \sum_{v \in J} x_v$. Then, we can write $f$ as
\[
f = \ell + L_I(-\ell + L_J) + g,
\] 
for some polynomial $g$ depending only on the variables $x_v$ with $v \in J$. Let $\xi_1, \ldots, \xi_s$ be independent $\Ber(p)$ random variables, and write $\vec{\xi}_I = (\xi_v)_{v \in I}$, $\vec{\xi}_J = (\xi_v)_{v \in J}$. Then, conditioning on the outcome of $\vec{\xi}_J$, we have
\begin{equation} \label{eq:conditioning-2}
\prob{f(\vec{\xi}_I, \vec{\xi}_J) = 0} = \mathbb E_{\vec{\xi}_J}\Big[{\probs{\vec{\xi}_I}{\ell + L_I(\vec{\xi}_I) \cdot (-\ell + L_J(\vec{\xi}_J)) + g(\vec{\xi}_J) = 0 \;|\; \vec{\xi}_J}}\Big].
\end{equation}
Note that if $L_J(\vec{\xi}_J) \neq \ell$, then the inner conditional probability is always at most
\[
\sup_{c \in \RR} \probs{\vec{\xi}_I}{L_I(\vec{\xi}_I) = c\;|\; \vec{\xi}_J} = \binmax{|I|}{p}. 
\]
On the other hand, since $\ell \neq 0$, 
\[
\probs{\vec{\xi}_J}{L_J(\vec{\xi}_J) = \ell} \le \binmaxplus{|J|}{p}.
\]
Therefore, we can bound the expression from \cref{eq:conditioning-2} as
\begin{align*}
\prob{f(\vec{\xi}(p)) = 0} &\le \prob{L_J(\vec{\xi}_J) \neq \ell} \cdot \binmax{|I|}{p} + \prob{L_J(\vec{\xi}_J) = \ell} \\
&\le \big(1 - \binmaxplus{|J|}{p}\big) \cdot \binmax{|I|}{p} + \binmaxplus{|J|}{p}.
\end{align*}
This gives the desired bound.
\end{claimproof}

Suppose that $G$ is a complete multipartite graph with a part $I$ of size at least $2$. Recalling  \cref{eq:0.475}, for every $m \ge 0$ we have $\binmaxplus{m}{p} < 0.475$, and for every $m \ge 2$ we have $\binmax{m}{p} < 0.475$. Therefore, \cref{complete-multipartite} implies that in this case
\[
\prob{f(\vec{\xi}(p)) = 0} \le 1 - (1 - 0.475)(1 - 0.475) < 0.725.
\]
So, it remains to consider the case when $G$ is a complete graph. In this case, for independent $\Ber(p)$ random variables $\xi_1, \ldots, \xi_s$, we have
\[
f(\xi_1, \ldots, \xi_s) = \ell - \ell \cdot (\xi_1 + \ldots + \xi_s) + \sum_{1 \le i < j \le s} \xi_i \xi_j = \frac{1}{2}(\xi_1 + \ldots + \xi_s - 2\ell)(\xi_1 + \ldots + \xi_s - 1).
\]
Therefore, $f(\xi_1, \ldots, \xi_s) = 0$ if and only if $\xi_1 + \ldots + \xi_s \in \{1, 2\ell\}$, and it suffices to verify the following claim.
\begin{claim} \label{two-layers}
For any nonzero real numbers $\ell_1, \ell_2$, we have $\prob{\xi_1 + \ldots + \xi_s \in \{\ell_1, \ell_2\}} < 0.713$.
\end{claim}
\begin{claimproof}
The proof is a routine calculation. Namely, the cases $s = 1, 2, 3$ can be checked directly, and for $s \ge 4$ we have
\[
\prob{\xi_1 + \ldots + \xi_s \in \{\ell_1, \ell_2\}} \le 2 \cdot \binmax{s}{p} \le 2 \cdot \binmax{4}{p} < 0.713. \qedhere
\]
\end{claimproof}
\end{proof}

\section{$\ell$ is close to zero}
\label{sec:close-to-zero}

In this section, we complete the proofs of \cref{better-than-1/e,2/e^2} by dealing with the case when $\ell = O(1)$. First, in \cref{reduction-to-bernoulli}, we reduce our problem to a question about polynomials of Bernoulli random variables. Then, in \cref{reduction-to-finite}, we further reduce this question to a finite verification.

It turns out that when $\ell = O(1)$ there is essentially no difference (for our edge-inducibility problem) between considering a uniformly random $k$-subset of vertices $X$ and just including each vertex independently with probability $k/n$. Indeed, intuitively, if the host graph $G$ has very few edges, then most of its vertices are isolated and we can use \cref{product-slice} to pass from the slice to independent Bernoulli random variables. On the other hand, if $G$ has at least a moderate number of edges, then the probability that $e(G[X]) = \ell$ is automatically small. The following lemma makes this intuition precise.

\begin{lemma} \label{reduction-to-bernoulli}
    Fix any $p \in (0, 1/2]$. Then, for each $\ell$ such that $1 \le \ell \le k$,  
    \[
    \ind(k, \ell) \le \sup_{f}\prob{f(\vec{\xi}(p)) = \ell} + O_p\left(\frac{\ell \log k + \log^2 k}{k}\right),
    \]
    where the supremum is over all homogeneous multilinear quadratic polynomials $f$ with coefficients in $\{0, 1\}$.
\end{lemma}
\begin{proof}
    Set $R = 1/p$, and assume that $k$ is sufficiently large in terms of $R$. It is sufficient to obtain an upper bound on $N(n, k, \ell)/\binom{n}{k}$ for $n = Rk$. So, we need to prove that for every graph $G$ on $n = Rk$ vertices and for a random $k$-subset $X \subseteq V(G)$, we have
    \begin{equation} \label{eq:version-with-R}
    \prob{e(G[X]) = \ell} \le \sup_{f}\prob{f(\vec{\xi}(p)) = \ell} + O_R\left(\frac{\ell \log k + \log^2 k}{k}\right).
    \end{equation}

    First, we clean up the graph by removing edges adjacent to high-degree vertices. Specifically, define
    \[
    U = \{v \in V(G) : \deg_G(v) \ge 2 R \ell + 2^6 R^2 \log k\},
    \]
    and let $G'$ be the graph obtained from $G$ by deleting all edges adjacent to vertices of $U$. Then, for each $u \in U$, 
    \[
    \mathbb E\Big[|X \cap N_G(u)|\Big] = \frac{\deg_G(u)}{R} \ge \ell + \frac{\deg_G(u)}{2R}.
    \]
    Hence, by our concentration inequality on the slice (\cref{concentration-on-the-slice}),
    \[
    \mathbb P\Big[|X \cap N_G(u)| \le \ell\Big] \le 2 \exp\left(-\frac{(\deg_G(u)/(2R))^2}{8\deg_G(u)}\right) \le 2 \exp(-2 \log k) \le \frac{2}{k^2}.
    \]
    Taking the union bound over $u \in U$, we conclude that with probability at least $1 - 2n/k^2$, for every $u \in U$ we have $|X \cap N_G(u)| > \ell$.
    However, in this case, if $e(G[X]) = \ell$ then $X$ must be disjoint from $U$, and, in particular, we must have $G[X] = G'[X]$. Therefore,
    \[
    \prob{e(G[X]) = \ell} \le \prob{e(G'[X]) = \ell} + 2n/k^2 = \prob{e(G'[X]) = \ell} + 2R/k,
    \]
    and hence it suffices to prove \cref{eq:version-with-R} with $G'$ in place of $G$.
    
    We can interpret $e(G'[X])$ as a homogeneous quadratic polynomial of $\vec{\sigma} \sim \Slice(n, k)$:
    \[
    e(G'[X]) = f_{G'}(\vec{\sigma}), \text{ where } f_{G'}((\sigma_v)_{v \in V(G')}) = \sum_{uv \in E(G')} \sigma_u \sigma_v.
    \]
    \textbf{Case 1:} $e(G') \le 2^8 R^4 \log k \cdot (2R\ell + 2^6 R^2 \log k) = O_R(\ell \log k + \log^2 k)$. Since $f_{G'}$ depends on at most $2e(G')$ variables, by \cref{product-slice} we conclude that
    \[
    \prob{f_{G'}(\vec{\sigma}) = \ell} \le \prob{f_{G'}(\vec{\xi}(p)) = \ell} + \max(2e(G')/n, 3/k) = \prob{f_{G'}(\vec{\xi}(p)) = \ell} + O_R\left(\frac{\ell \log k + \log^2 k}{k}\right),
    \]
    which gives the desired bound.

    \textbf{Case 2:} $e(G') > 2^8 R^4 \log k \cdot (2R\ell + 2^6 R^2 \log k)$. Since we have cleaned up high-degree vertices, each vertex of $G'$ has degree at most $2R\ell + 2^6 R^2 \log k$. Then,
    \[
    \sum_{v \in V(G')} \deg_{G'}(v)^2 \le (2R\ell + 2^6 R^2 \log k) \cdot 2 e(G').
    \]
    The expected number of edges in $G'[X]$ satisfies
    \[
    \expected{e(G'[X])} = \frac{k(k-1)}{n(n-1)} \cdot e(G') \ge \frac{e(G')}{2R^2} \ge \ell + \frac{e(G')}{4R^2}.
    \]
    Applying our concentration inequality on the slice (\cref{concentration-on-the-slice}) with $a_v = \deg_{G'}(v)$, we conclude that
    \begin{align*}
    \prob{e(G'[X]) = \ell} &\le 2\exp\left(-\frac{(e(G')/(4R^2))^2}{8\sum\limits_{v \in V(G')} \deg_{G'}(v)^2}\right)
    \le 2\exp\left(-\frac{e(G')}{2^8 R^4 (2 R \ell + 2^6 R^2 \log k)}\right) \\
    &\le 2\exp(-\log k) \le 2/k.
    \end{align*}
    This bound is certainly smaller than the one required in \cref{eq:version-with-R}, thus we are done.
\end{proof}
Recall from \cref{eq:binmax-definition} that for each $m \in \NN$ and $p \in [0, 1]$, we write
\[
\binmax{m}{p} = \max_{0 \le m_0 \le m}\prob{\Bin(m, p) = m_0} = \max_{0 \le m_0 \le m} \binom{m}{m_0} p^{m_0} (1-p)^{m-m_0}.
\]
The following lemma states that $\binmax{m}{p}$ is also an upper bound on the point concentration probability for a general class of polynomials.
\begin{lemma} \label{large-linear-part}
    Let $f$ be a polynomial with non-negative coefficients, which has at least $m$ nonzero linear terms. Then, for each $\ell \in \RR$,
    \[
    \prob{f(\vec{\xi}(p)) = \ell} \le \binmax{m}{p}.
    \]
\end{lemma}
\begin{proof}
    Let $f$ be a polynomial in $s$ variables $x_1, \ldots, x_s$, and let $\xi_1, \ldots, \xi_s$ be independent $\Ber(p)$ random variables. Then, we need to prove that
    \[
    \prob{f(\xi_1, \ldots, \xi_s) = \ell} \le \binmax{m}{p}.
    \]  
    Without loss of generality, suppose that monomials $x_1, \ldots, x_m$ appear in $f$ with positive coefficients. Conditioning on the outcomes of $(\xi_i)_{i > m}$, we may assume that $s = m$.

    For a set $A \subseteq [m]$, let $\mathds{1}_A \in \{0, 1\}^m$ denote its characteristic vector. Note that the polynomial $f$ is strictly increasing (on $\{0, 1\}^m$) with respect to each of its $m$ variables. Therefore, the family $\mc{A}$ of subsets of $[m]$ defined as
    \[
    \mc{A} = \{A \subseteq [m] : f(\mathds{1}_A) = \ell \}
    \]
    is an antichain. Using the BLYM inequality (\cref{BLYM}), we conclude that
    \begin{align*}
    \prob{f(\xi_1, \ldots, \xi_m) = \ell} &= \prob{\{i \in [m] : \xi_i = 1\} \in \mc{A}} = \sum_{A \in \mc{A}} p^{|A|}(1-p)^{m-|A|} \\
    &\le \max_{A \in \mc{A}} \binom{m}{|A|} p^{|A|}(1-p)^{m-|A|} \le \binmax{m}{p}.
    \end{align*}
    This completes the proof.
\end{proof}

Now, we explain how to obtain bounds on $\ind(k,\ell)$ (for $\ell=O(1)$) by a finite computation.
\begin{definition}
    Let $\mc{G}_s$ be the set of multilinear quadratic polynomials in $s$ variables with coefficients in $\{0, 1\}$ and no constant term, and let $\mc{G} = \bigsqcup_{s \ge 1} \mc{G}_s$. Next, let $\mc{G}(m) \subseteq \mc{G}$ consist of polynomials $g \in \mc{G}$ satisfying the following condition:
    for every variable $x_i$ of $g$, the polynomial $g_i$ obtained from $g$ by the substitution $x_i = 1$ does not belong to $\mc{G}$ and has fewer than $m$ nonzero linear terms.
\end{definition}

\begin{theorem} \label{reduction-to-finite}
    Consider a polynomial $f \in \mc{G}$. Then, for each $m, \ell \in \NN$ and $p \in [0, 1]$, we have
    \begin{equation} \label{eq:rho-of-Gn}
    \prob{f(\vec{\xi}(p)) = \ell} \le \max\left(\binmax{m}{p}, \max_{g \in \mc{G}(m)}\prob{g(\vec{\xi}(p)) = \ell}\right).
    \end{equation}
\end{theorem}

Importantly, $\mc{G}(m)$ is a finite set: as we will show in \cref{Gn-is-small}, every element of $\mc{G}(m)$ is a polynomial in at most $(m+1)^2/4$ variables. Therefore, \cref{reduction-to-finite} provides us with a computational approach to obtaining upper bounds on the concentration probability $\prob{f(\vec{\xi}(p)) = \ell}$ for an arbitrary $f \in \mc{G}$. Indeed, after fixing certain $m \in \NN$, the right hand side of \cref{eq:rho-of-Gn} becomes a finite and relatively explicit expression, which can be then optimised over $p \in [0, 1]$. Taking larger $m$ here should yield better bounds, but, as the size of $\mc{G}(m)$ grows, this increases the amount of computation needed.

\begin{proof}[\textbf{Proof of \cref{reduction-to-finite}}]
    We prove the statement for every $f \in \mc{G}_s$ by induction on $s$. If $f$ lies in $\mc{G}(m)$, then we have nothing to prove. Otherwise, there exists a variable $x_i$ such that the polynomial $f^{(1)}_i$ obtained from $f$ by the substitution $x_i = 1$ either belongs to $\mc{G}$ or has at least $m$ nonzero linear terms.

    Note that the polynomial $f^{(0)}_i$ obtained from $f$ by the substitution $x_i = 0$ always belongs to $\mc{G}_{s-1}$ (indeed, it depends on at most $s-1$ variables, and the set of its nonzero coefficients is a subset of the nonzero coefficients of $f$). Thus, by the induction hypothesis, we have
    \[
    \prob{f^{(0)}_i(\vec{\xi}(p)) = \ell} \le \max\left(\binmax{m}{p}, \max_{g \in \mc{G}(m)}\prob{g(\vec{\xi}(p)) = \ell}\right).
    \]
    Next, we observe that an analogous bound holds for $f^{(1)}_i$ as well. Indeed, if $f^{(1)}_i$ belongs to $\mc{G}_{s-1}$, then we can again use the induction hypothesis. Otherwise, it has at least $m$ nonzero linear terms. While its coefficients do not necessarily lie in $\{0, 1\}$, they are still non-negative. Hence, by \cref{large-linear-part}, we have $\prob{f^{(1)}_i(\vec{\xi}(p)) = \ell} \le \binmax{m}{p}$. In both cases, we can conclude that
    \[
    \prob{f^{(1)}_i(\vec{\xi}(p)) = \ell} \le \max\left(\binmax{m}{p}, \max_{g \in \mc{G}(m)}\prob{g(\vec{\xi}(p)) = \ell}\right).
    \]
    This completes the proof, since
    \[
    \prob{f(\vec{\xi}(p)) = \ell} = p \cdot \prob{f^{(1)}_i(\vec{\xi}(p)) = \ell} + (1-p) \cdot \prob{f^{(0)}_i(\vec{\xi}(p)) = \ell}. \qedhere
    \]
\end{proof}

\begin{lemma} \label{Gn-is-small}
    Every $g \in \mc{G}(m)$ is a polynomial in at most $(m+1)^2/4$ variables (that is, $g \in \mc{G}_{s}$ for some $s \le (m+1)^2/4$). 
\end{lemma}
\begin{proof}
    Let $g$ be a polynomial in $s$ variables $x_1, \ldots, x_s$. Recall that, by definition of $\mc{G}(m)$, for each variable $i \in [s]$, the polynomial $g_i$ obtained from $g$ by the substitution $x_i = 1$ 
    \begin{itemize}
        \item does not belong to $\mc{G}$ (that is, it either has a nonzero constant term, or has a coefficient at least $2$), and
        \item has at most $m-1$ nonzero linear terms.
    \end{itemize}
    Let 
    \[
    L = \{i \in [s] : g \text{ has a term } x_i\}, \quad Q = [s] \setminus L.
    \]
    Note that for each $j \in Q$, there exists $i \in L$ such that $g$ contains a term $x_i x_j$: otherwise, we would have $g_j \in \mc{G}$. On the other hand, for each $i \in L$ there are at most $m-|L|$ indices $j \in Q$ such that $g$ contains a term $x_i x_j$: otherwise, $g_i$ would have at least $(m-|L|+1) + (|L|-1) = m$ nonzero linear terms. Together, these two statements imply that $|Q| \le |L| \cdot (m - |L|)$, and thus
    \[
    s = |L| + |Q| \le |L| \cdot (m - |L| + 1) \le \frac{(m+1)^2}{4}. \qedhere
    \]
\end{proof}

Next, we deduce \cref{0.33,0.27}, which will be used to complete the proofs of \cref{better-than-1/e,2/e^2}, respectively.

\begin{proposition} \label{0.33}
    There exists $p \in (0, 1/2]$, such that for every polynomial $f \in \mc{G}$ and every integer $\ell \ge 2$ we have 
    \[
    \prob{f(\vec{\xi}(p)) = \ell} < 0.3293 < 0.33.
    \] 
\end{proposition}

\begin{proof}
    By \cref{reduction-to-finite}, it is sufficient to find $m \in \NN$ and $p \in (0, 1/2]$ such that
    \[
    \max\left(\binmax{m}{p}, \max_{g \in \mc{G}(m), \ell \ge 2}\prob{g(\vec{\xi}(p)) = \ell}\right) < 0.3293.
    \] 
    We take $m = 5$ and $p = 1/3$. Note that $\mc{G}(5)$ is relatively small: by \cref{Gn-is-small}, it contains only polynomials in at most $9$ variables such that (by definition) each variable appears in at most $4$ of its quadratic terms. For each of these finitely many polynomials, one needs to check only finitely many values of $\ell$. Therefore, calculation of $\max\limits_{g \in \mc{G}(5), \ell \ge 2}\prob{g(\vec{\xi}(p)) = \ell}$ is within reach of a computer program\footnote{Our (quite straightforward) code is available in the ancillary files of the arXiv submission.}. This way, we obtain that
    \[
    \binmax{5}{p} = \max_{g \in \mc{G}(5), \ell \ge 2}\prob{g(\vec{\xi}(p)) = \ell} = 0.3292... < 0.3293.
    \]
    (the maximum is attained at $\ell = 2$ and the polynomial $(1 + x_1)(x_2 + x_3 + x_4 + x_5) \in \mc{G}(5)$). This completes the proof.
\end{proof}

\begin{remark}
\label{computational_data}

The table below depicts, for $m = 2, 3, 4, 5$, the sizes of $\mc{G}(m)$ (after identifying polynomials that can be obtained from each other by a permutation of variables), the optimal value of $p$ one needs to take in the above argument, and the resulting bound on $\sup\limits_{f \in \mc{G}, \ell \ge 2}\prob{f(\vec{\xi}(p)) = \ell}$.

\begin{center}
\begin{tabular}{|l|c|c|c|c|}
\hline
$\boldsymbol{m}$ & \textbf{2} & \textbf{3} & \textbf{4} & \textbf{5} \\
\hline
$|\mc{G}(m) / \sim |$ & 4 & 16 & 99 & 1653 \\
\hline
optimal value of $p$ & 2/3 \footnote{Some results in this section are stated only for $p \le 1/2$, but this assumption is not necessary and can be easily removed.} & 0.5 & 0.4 & 1/3 \\
\hline
bound on $\sup\limits_{f \in \mc{G}, \ell \ge 2}\prob{f(\vec{\xi}(p)) = \ell}$ & 0.444... & 0.375 & 0.3456 & 0.3292... \\
\hline
\end{tabular}
\end{center}

It is plausible that this bound could be improved by considering larger values of $m$, perhaps with a more optimised search algorithm or greater computational resources. However, it is not clear whether this approach can yield the optimal bound of $2/e^2 + o(1)$ for $\ell = 2$ and $p \to 0$.
\end{remark}

\begin{proposition} \label{0.27}
    There exists $p \in (0, 1/2]$ such that for every polynomial $f \in \mc{G}$ and every integer $\ell \ge 60$ we have     
    \[
    \prob{f(\vec{\xi}(p)) = \ell} < 0.27 < 2/e^2.
    \]
\end{proposition}
\begin{proof}
    By \cref{reduction-to-finite}, it is sufficient to find $m \in \NN$ and $p \in (0, 1/2]$ such that
    \[
    \max\left(\binmax{m}{p}, \max_{g \in \mc{G}(m), \ell \ge 2}\prob{g(\vec{\xi}(p)) = \ell}\right) < 0.27.
    \] 
    We take $m = 8$ and $p = 0.426$. A simple calculation shows that $\binmax{8}{0.426} < 0.27$.

    Now consider a polynomial $g \in \mc{G}(8)$. By \cref{Gn-is-small}, it is a polynomial in at most $\lfloor(m+1)^2/4\rfloor = 20$ variables. Also, by the definition of $\mc{G}(m)$, it has at most $m = 8$ linear terms, and each of its variables appears in at most $m-1 = 7$ quadratic terms. Therefore, 
    \[
    \expected{g(\vec{\xi}(p))} \le p^2 \cdot \frac{7 \cdot 20}{2} + p \cdot 8 < 16.112.
    \]
    Since $\ell \ge 60$, by Markov's inequality we conclude that
    \[
    \prob{g(\vec{\xi}(p)) = \ell} \le \frac{1}{\ell} \expected{g(\vec{\xi}(p))} < 0.27. 
    \qedhere
    \]
\end{proof}

Finally, we can formally deduce \cref{better-than-1/e,2/e^2} from \cref{bulk,far-from-multiple,close-to-multiple,reduction-to-bernoulli,0.33,0.27}.

\begin{proof}[\textbf{Proofs of \cref{better-than-1/e,2/e^2}}]
    For \cref{2/e^2}, we fix an arbitrary $\eps > 0$; for \cref{better-than-1/e}, we set $\eps = 0.01$. 
    First, we prove that $\ind(k, \ell) \le 2/e^2 + \eps$ when $C < \ell \le \frac{1}{2}\binom{k}{2}$ (for some absolute constant $C$), $\ell \neq k-1$, and $k$ is sufficiently large in terms of $\eps$.
    \begin{itemize}
        \item When $\ell > C_1 k$ (for a certain absolute constant $C_1$) the desired bound follows from \cref{bulk};
        \item When $\ell \le C_1 k$ and does not lie in an interval of the form $[ak- C, ak + C]$ for some positive integer $a \ge 1$ (for a certain absolute constant $C$), the bound follows from \cref{far-from-multiple};
        \item When $\ell$ lies in some interval $[ak - C, ak + C]$ with $1 \le a \le C_1 + 1$, the bound follows from \cref{close-to-multiple} (here we use that $\ell \neq k-1$).
    \end{itemize}
    It remains to deal with the case when $\ell \le C$. By \cref{reduction-to-bernoulli}, for every $p \in (0, 1/2]$ we have
    \[
    \ind(k, \ell) \le \sup_{f \in \mc{G}} \prob{f(\vec{\xi}(p)) = \ell} + O_p\left(\frac{\ell \log k + \log^2 k}{k}\right).
    \]
    Note that for a fixed $p$ and $\ell \le C$, the second term tends to zero as $k \to \infty$.
    \cref{0.33} states that there exists $p \in (0, 1/2]$ such that for every $f \in \mc{G}$ and $\ell \ge 2$, we have $\prob{f(\vec{\xi}(p)) = \ell} < 0.3293 < 0.33$. This implies \cref{better-than-1/e}.
    Similarly, \cref{0.27} states that there exists $p \in (0, 1/2]$ such that for every $f \in \mc{G}$ and $\ell \ge 60$, we have $\prob{f(\vec{\xi}(p)) = \ell} < 0.27 < 2/e^2$. This implies \cref{2/e^2}.
\end{proof}

\section{Hypergraphs}
\label{sec:hypergraphs}

In this section, we study edge-inducibilities in hypergraphs and prove \cref{hypergraphs-weak}. In fact, we deduce it from a structural result about polynomials on the slice (\cref{polynomial-dichotomy} below). Roughly speaking, this result states that if a polynomial on the slice is poorly anticoncentrated, then one can make it constant by fixing values of a small number of its variables.

\begin{definition} \label{def:constantly-ell}
    We say that a multilinear polynomial $f$ in $n$ variables of degree at most $r$ is \emph{constantly $\ell$ on $\Slice(n, k)$} if for every $0 \le d \le r$, all $\binom{n}{d}$ coefficients of its degree-$d$ monomials are equal to some $\ell_d \in \RR$, and
    \[
    \sum_{d = 0}^r \ell_d \binom{k}{d} = \ell.
    \]
    In particular, this implies that $f(\vec{x}) = \ell$ for every $\vec{x} \in \Slice(n, k)$.
\end{definition}

For a polynomial $f$ in $n$ variables and disjoint sets $Y_0, Y_1 \subseteq [n]$, let $f_{Y_0, Y_1}$ denote the polynomial obtained from $f$ by setting all variables in $Y_0$ to zero and all variables in $Y_1$ to one. Note that if $f$ is constantly $\ell$ on $\Slice(n, k)$ (and $|Y_1| \le k$), then $f_{Y_0, Y_1}$ is also constantly $\ell$ on $\Slice(n - |Y_0| - |Y_1|, k - |Y_1|)$.

\begin{theorem} \label{polynomial-dichotomy}
    Consider $r, q, R, n, k \in \NN$ and $K_0, \eps > 0$ such that $3k - K_0 \le n \le Rk + K_0$, and let $f$ be an $n$-variable multilinear polynomial of degree at most $r$ with coefficients in $\ZZ\cap[-q, q]$. Then, for each $\ell \in \RR$, at least one of the following holds:
    \begin{enumerate}
        \item[1.] $\prob{f(\vec{\sigma}) = \ell} \le \eps$, where $\vec{\sigma} \sim \Slice(n, k)$;
        \item[2.] There exist $C' = C'(\eps, r, q, R, K_0)$ and disjoint sets $Y_0, Y_1 \subseteq [n]$ of size at most $C'$, such that $f_{Y_0, Y_1}$ is constantly $\ell$ on $\Slice(n - |Y_0| - |Y_1|, k - |Y_1|)$ in the sense of \cref{def:constantly-ell}.
    \end{enumerate}    
\end{theorem}

The proof of \cref{polynomial-dichotomy} relies on the following two propositions. \cref{sparse-polynomials-general} extends \cite[Lemma 5.1]{jain-kwan-mubayi-tran-25} (stated in this paper as \cref{large-matching}) to sparse polynomials on the slice with possibly negative coefficients. \cref{dense-polynomials} provides a suitable analogue of it for dense polynomials on the slice, generalising \cite[Lemma 4.1]{jain-kwan-mubayi-tran-25}. 

\begin{proposition} \label{sparse-polynomials-general}
    For $r, q \in \NN$, there exists $\delta = \delta(r, q) > 0$ such that the following holds. Consider $R, k, m, n \in \NN$ such that $2k \le n \le Rk$, and let $f$ be an $n$-variable multilinear polynomial of degree at most $r$ with coefficients in $\ZZ \cap [-q, q]$. Suppose that $\nu_r(f) \ge m$ and $f$ has at most $\delta \binom{n}{r}$ nonzero degree-$r$ terms. Then, for $\vec{\sigma} \sim \Slice(n, k)$,
    \[
    \sup_{\ell \in \RR} \prob{f(\vec{\sigma}) = \ell} \lesssim_{r, q, R} \max_{r' \le r} \LO_{r'}(\Omega_{r, q, R}(m)).
    \]
\end{proposition}

\begin{proposition} \label{dense-polynomials}
    Consider $r, q, R, n, k \in \NN$ and $\delta > 0$ such that $2k \le n \le Rk$. Let $f$ be an $n$-variable multilinear polynomial of degree at most $r$ with coefficients in $\ZZ\cap[-q, q]$, such that no $(1-\delta)\binom{n}{r}$ of its degree-$r$ coefficients are equal to the same value. Then, for $\vec{\sigma} \sim \Slice(n, k)$,
    \[
    \sup_{\ell \in \RR} \prob{f(\vec{\sigma}) = \ell} \lesssim_{r, q, R} \max_{r' \le r} \LO_{r'}(\Omega_{r, q, R}(\delta k)).
    \]
\end{proposition}

Our proofs of \cref{sparse-polynomials-general,dense-polynomials} are minor adaptations of the arguments in \cite{jain-kwan-mubayi-tran-25}, and are deferred to \cref{sec:appendix}.

\begin{proof}[\textbf{Proof of \cref{polynomial-dichotomy} assuming \cref{sparse-polynomials-general,dense-polynomials}}]
    Our proof goes by induction on $r$ (in the base case $r = 0$, the polynomial is clearly constant on the slice). We may assume that $k$ is sufficiently large in terms of $\eps, r, q, R, K_0$ (otherwise, the statement holds trivially). In particular, we assume that $2k \le n \le (R+1)k$.

    Let $\delta = \delta(r, q) > 0$ be the constant from \cref{sparse-polynomials-general}. 

    \textbf{Case 1:} no $(1-\delta)\binom{n}{r}$ of the degree-$r$ coefficients of $f$ are equal to the same value. 
    Recall from the discussion after \cref{def:LO} that for every $r' \le r$ and $N \ge 1$
    \begin{equation} \label{eq:LO-application}
    \LO_{r'}(N) \le (\log N)^{O_r(1)} / \sqrt{N} \lesssim_{r} N^{-1/3}.
    \end{equation}
    Then, by \cref{dense-polynomials}, we have
    \[
    \prob{f(\vec{\sigma}) = \ell} \lesssim_{r, q, R} \max_{r' \le r} \LO_{r'}(\Omega_{r, q, R}(\delta k)) \lesssim_{r, q, R} (\delta k)^{-1/3}.
    \]
    Since $k$ is sufficiently large in terms of $\eps, r, q, R$, we conclude that $\prob{f(\vec{\sigma}) = \ell} \le \eps$. This completes the proof in this case.
    
    \textbf{Case 2:} $(1 - \delta) \binom{n}{r}$ degree-$r$ coefficients of $f$ are equal to some $\ell_r \in \ZZ \cap [-q, q]$. Then, the polynomial $f^*$ obtained from $f$ by subtracting $\ell_r$ from each of its degree-$r$ coefficients, has at most $\delta \binom{n}{r}$ nonzero degree-$r$ terms. Note that the coefficients of $f^*$ lie in $\ZZ \cap [-2q, 2q]$, and denote $\ell^* = \ell - \ell_r \binom{k}{r}$. By \cref{sparse-polynomials-general} (combined with \cref{eq:LO-application}), we conclude that
    \[
    \prob{f(\vec{\sigma}) = \ell} = \prob{f^*(\vec{\sigma}) = \ell^*} \lesssim_{r, q, R} \max_{r' \le r} \LO_{r'}(\Omega_{r, q, R}(\nu_r(f^*))) \lesssim_{r, q, R} (\nu_r(f^*))^{-1/3}.
    \]
    Thus, there exists a constant $C_0 = C_0(r, q, R, \eps)$ such that if $\nu_r(f^*) \ge C_0/r$, then $\prob{f(\vec{\sigma}) = \ell} \le \eps$. So, we may assume that $\nu_r(f^*) < C_0/r$. In this case, there exists a set $U \subseteq [n]$ of size at most $C_0$ such that every degree-$r$ monomial of $f^*$ contains at least one variable from $U$.

    Consider a partition $U = U_0 \sqcup U_1$. Recall that $f^*_{U_0, U_1}$ denotes the polynomial obtained from $f^*$ by setting all variables in $U_0$ to zero and all variables in $U_1$ to one. Note that $f^*_{U_0, U_1}$ has degree at most $r-1$, that its coefficients lie in $\ZZ \cap [-2qrC_0^r, 2qrC_0^r]$, and that
    \[
    3(k-|U_1|)-K_0-C_0 \le n-|U_0|-|U_1| \le R(k-|U_1|)+K_0+RC_0.
    \]
    Therefore, we can apply the induction hypothesis to $f^*_{U_0, U_1}$ and $\ell^*$. Suppose that statement 2 holds for $f^*_{U_0, U_1}$ and $\ell^*$: that is, there exist disjoint sets $Y'_0, Y'_1 \subseteq [n] \setminus U$ such that $f^*_{U_0 \cup Y'_0, U_1 \cup Y'_1}$ is constantly $\ell^*$ on the corresponding slice. Let $Y_0 = U_0 \cup Y'_0$ and $Y_1 = U_1 \cup Y'_1$. Clearly, $f-f^*$ is constantly $\ell_r\binom{k}{r}$ on $\Slice(n, k)$, and hence $(f-f^*)_{Y_0, Y_1}$ is constantly $\ell_r\binom{k}{r}$ on $\Slice(n - |Y_0| - |Y_1|, k - |Y_1|)$. Therefore, $f_{Y_0, Y_1} = f^*_{Y_0, Y_1} + (f - f^*)_{Y_0, Y_1}$ is constantly $\ell = \ell^* + \ell_r\binom{k}{r}$ on $\Slice(n - |Y_0| - |Y_1|, k - |Y_1|)$. This means that statement 2 also holds for $f$ and $\ell$.

    So, we may assume that, for every partition $U = U_0 \sqcup U_1$, statement 1 of \cref{polynomial-dichotomy} holds for $f^*_{U_0, U_1}$ and $\ell^*$. In this case, we claim that statement 1 holds for $f$ and $\ell$ as well. Indeed, let $U'_0 = \{i \in U : \sigma_i = 0\}$ and $U'_1 = \{i \in U : \sigma_i = 1\}$ be the random partition of $U$ induced by $\vec{\sigma} \sim \Slice(n, k)$. Then, conditioning on the outcome of $U'_0$ and $U'_1$, we have
    \[
    \prob{f(\vec{\sigma}) = \ell} = \prob{f^*(\vec{\sigma}) = \ell^*} = \mathbb{E}_{U'_0, U'_1}\left[\prob{f^*_{U'_0, U'_1}(\vec{\sigma}) = \ell^* \;|\; U'_0, U'_1}\right] \le \expecteds{U'_0, U'_1}{\eps} = \eps. 
    \]
    This completes the proof.
\end{proof}

\begin{proof}[\textbf{Proof of \cref{hypergraphs-weak}}]    
    We may assume that $k$ is sufficiently large in terms of $\eps, r$ (otherwise, the statement holds trivially). Recall that $\ind_r(k, \ell) > \eps$ is the limit of a non-increasing sequence $N_r(n, k, \ell) / \binom{n}{k}$ as $n \to \infty$. In particular, $N_r(n, k, \ell) > \eps \binom{n}{k}$ for $n = 3k$, and hence there exists an $r$-uniform hypergraph on $n = 3k$ vertices such that a uniformly random $k$-subset of its vertices induces exactly $\ell$ edges with probability greater than $\eps$. Interpreting this hypergraph as a homogeneous multilinear degree-$r$ polynomial $f$ with coefficients in $\{0, 1\}$, we have $\prob{f(\vec{\sigma}) = \ell} > \eps$ for $\vec{\sigma} \sim \Slice(n, k)$.

    Therefore, by \cref{polynomial-dichotomy} (applied with $q = 1$, $R = 3$), there exist $C' = C'(\eps, r)$ and disjoint sets $Y_0, Y_1 \subseteq [n]$ of size at most $C'$ such that the polynomial $f_{Y_0, Y_1}$ (obtained from $f$ by setting all variables in $Y_0$ to zero and all variables in $Y_1$ to one) is constantly $\ell$ on $\Slice(n - |Y_0| - |Y_1|, k - |Y_1|)$. This means that for every $0 \le d \le r$, all degree-$d$ coefficients of $f_{Y_0, Y_1}$ are equal to some $\ell'_d \in \RR$, and
    \begin{equation} \label{eq:linear-combination-with-Y_1}
    \ell = \sum_{d = 0}^r \ell'_d \binom{k - |Y_1|}{d}.
    \end{equation}
    Furthermore, since $f$ is a homogeneous multilinear polynomial of degree $r$ with coefficients in $\{0, 1\}$, the coefficients of $f_{Y_0, Y_1}$ are non-negative integers at most $(C')^r$. Hence, $\ell'_0, \ldots, \ell'_r \in \{0, \ldots, (C')^r\}$.
    Then, using the following elementary identities (valid for all integers $a, b, c \ge 0$, $a \ge b+c$)
    \[
    \binom{a+c}{b} = \sum_{i=0}^{b} \binom{c}{i} \binom{a}{b-i}, \qquad \binom{a-c}{b} = \sum_{i=0}^{b} (-1)^i \binom{c+i-1}{i} \binom{a}{b-i},
    \]
    we can rewrite \cref{eq:linear-combination-with-Y_1} as
    \[
    \ell = \sum_{d=0}^r \ell_{r-d} \binom{k-(r-d)}{d} = \sum_{d=0}^r \ell_d \binom{k-d}{r-d},
    \]
    where $\ell_0, \ldots, \ell_r \in \ZZ \cap [-C, C]$ and $C = r ((C'+r)C')^r$.
    Since $0 \le \ell \le \binom{k}{r}$ and $k$ is sufficiently large in terms of $C, r$, we must have $\ell_0 \in \{0, 1\}$. This completes the proof.
\end{proof}

\section{Concluding remarks} \label{sec:concluding}

We have proved a number of results going beyond the edge-statistics theorem, but there is still much scope for further research.

\subsection{Edge-statistics for set systems}\label{subsec:set-systems}

Recall that $\ind_r(k, \ell)$ is bounded by the maximum probability, over $r$-uniform hypergraphs on $n$ vertices, that a uniformly random $k$-subset of vertices induces exactly $\ell$ edges. When $k$ is sufficiently large in terms of $r, \ell$, and $n/k$, it is essentially equivalent to consider a set obtained by including each vertex independently with probability $k/n$ (see \cref{reduction-to-bernoulli} for a precise statement in the case $r = 2$).
Motivated by this, we suggest studying the following variant of edge-inducibilities for independent samples from set systems of unbounded uniformity: for an integer $\ell$ and $p \in (0, 1)$, we define
\begin{equation} \label{eq:ind_*}
\ind_*(\ell, p) = \sup_{\mc{A}, B} \prob{\#\{A \in \mc{A} : A \subseteq B_p\} = \ell},
\end{equation}
where the supremum is over all finite sets $B$ and families of non-empty subsets $\mc{A} \subseteq 2^B$, and $B_p$ denotes a $p$-random subset of $B$ (including each element with probability $p$ independently).
A standard subsampling argument shows that $\ind_*(\ell, p)$ is non-decreasing in $p$, and thus we can further define 
\[
\ind_*(\ell) = \lim_{p \to 0} \ind_*(\ell, p).
\]
Equivalently, one can interpret $\ind_*(\ell, p)$ as the maximum probability that the value of a multilinear polynomial (with coefficients in $\{0, 1\}$ and no constant term) at independent $\Ber(p)$ random variables equals $\ell$. Using this reformulation, we notice that $\ind_*(\ell) \le 1/e$ for every $\ell \ge 1$ by a result of Fox, Kwan, and Sauermann~\cite[Theorem 1.8]{fox-kwan-sauermann-21}. While this bound is sharp for $\ell = 1$, we conjecture that for $\ell \ge 2$ it can be improved as follows.
\begin{conjecture} \label{conjecture-set-systems}
    For every $\ell \ge 2$, we have $\ind_*(\ell) \le 2/e^2$.
\end{conjecture}
If true, this would be sharp for $\ell \in \{2, 3\}$, as one can take $n = \lceil 2/p \rceil$ and either $\mc{A} = \{A \subseteq [n] : |A| = 1\}$ or $\mc{A} = \{A \subseteq [n] : 1 \le |A| \le 2\}$. 

\cref{conjecture-set-systems} should be viewed as an ``arbitrary uniformity'' analogue of \cref{2/e^2} in the regime $\ell = O(1)$. Indeed, \cref{reduction-to-bernoulli} implies that for a fixed $\ell$ and $k \to \infty$, we have $\ind(k, \ell) \le \ind_*(\ell) + o(1)$, and thus a proof of this conjecture would (combined with the results of this paper) yield the optimal $2/e^2 + o(1)$ bound in \cref{better-than-1/e}. In fact, for this application, it would suffice to restrict the supremum in \cref{eq:ind_*} to families consisting of sets of size $2$ (a full resolution would be useful for generalising our results to edge-inducibilities in hypergraphs).

\subsection{Tightness of \texorpdfstring{\cref{bulk,far-from-multiple}}{}: constructions and discussion}
\label{subsec:lower-bounds}

The bound in \cref{bulk} is tight (up to a multiplicative constant factor) when $\ell$ has the form $a(k-a)$ for some integer $a \in [0, k/2]$. Jain, Kwan, Mubayi, and Tran~\cite{jain-kwan-mubayi-tran-25} conjectured that it can be improved for ``generic'' values of $\ell$: more precisely, that for a $1 - o(1)$ proportion (as $k \to \infty$) of values of $\ell$ in the range $[0, \binom{k}{2}]$, one has $\ind(k, \ell) \le k^{-1/2-\delta}$ for some absolute constant $\delta > 0$.

We suspect that a similar phenomenon might occur in the sparse regime. Namely, the following constructions show that for every fixed $a \ge 0$, the bound in \cref{far-from-multiple} is tight for $\Omega_a(\sqrt{k})$ different values of $\ell$ within the interval $[ak, (a+1)k]$.
\begin{itemize}
    \item Suppose that $1 \le \ell_0 \le k - \sqrt{k}$ (thus, the $\ell_0^{-1/4}$ term in the bound from \cref{far-from-multiple} dominates), and that $\ell_0 = \binom{m}{2} - a^2$ for some $m \in \NN$. Let $n$ be much larger than $k$, and consider the following graph $G$ on $n$ vertices. Pick two disjoint sets of vertices $A, M \subseteq [n]$ such that $|A| = an/k, |M| = mn/k$. Let $G$ contain all the edges between $A$ and $[n] \setminus A$, as well as all the edges inside $M$. Then, one can check that a random $k$-subset of its vertices contains exactly $a$ vertices of $A$ and exactly $m$ vertices of $M$ with probability $\Theta(1/\sqrt{(a+1)m})$. If both these events occur, then the subgraph induced by this set has exactly $a(k-a) + \binom{m}{2} = ak + \ell_0$ edges. So, 
    \[
    \ind(k, ak + \ell_0) \gtrsim \frac{1}{\sqrt{(a+1)m}} \gtrsim_a \frac{1}{\ell_0^{1/4}}. 
    \]
    \item Suppose that $k - \sqrt{k} \le \ell_0 \le k-1$ (thus, the $(k-\ell_0)^{-1/2}$ term dominates), and that $k-\ell_0 = (a+1)(m+a+1)$ for some $m \in \NN$. Similarly, pick two disjoint sets $A, M \subseteq [n]$, $|A| = (a+1)n/k, |M| = mn/k$. Let $G$ contain all the edges between $A$ and $[n] \setminus (A \cup M)$. Then, with probability $\Theta(1/\sqrt{(a+1)m})$, a random $k$-subset of its vertices contains exactly $a+1$ vertices of $A$ and exactly $m$ vertices of $M$, and thus exactly $(a+1)(k-(a+1)-m) = ak + \ell_0$ edges. So,     
    \[
    \ind(k, ak + \ell_0) \gtrsim \frac{1}{\sqrt{(a+1)m}} \gtrsim \frac{1}{(k-\ell_0)^{1/2}}. 
    \]
\end{itemize}

In particular, $\ind(k, \ell) \gtrsim \ell^{-1/4}$ for each $\ell \le k$ that is of the form $\ell = \binom{m}{2}$. On the other hand, for an arbitrary $\ell \le k/2$, the best construction we know comes from a decomposition $\ell = \binom{m_1}{2} + \ldots + \binom{m_s}{2}$: one can take the host graph to be a disjoint union of cliques of sizes $m_1 n/k, \ldots, m_s n/k$ (and isolated vertices) to conclude that $\ind(k, \ell) \ge (\prod_{i = 1}^s m_i)^{-1/2}$. However, for most values of $\ell$, one can only find such a decomposition with $\prod_{i = 1}^s m_i = \ell^{1 + o(1)}$.
Therefore, we conjecture that for $1 - o(1)$ proportion (as $k \to \infty$) of values of $\ell$ in the range $[0, k]$, one has $\ind(k, \ell) \le \ell^{-1/4-\delta}$ for some absolute constant $\delta > 0$.

\bibliographystyle{plain}

\bibliography{references}

\appendix \section{Proofs of \cref{sparse-polynomials-general,dense-polynomials}}
\label{sec:appendix}

Throughout this section, we denote the set of all subsets of some finite set $X$ that have size $r$ by $\binom{X}{r}$, and the set of all subsets of $X$ that have size at most $r$ by $\binom{X}{\le r}$. We also write $\alpha \gg (\beta_1, \ldots, \beta_M)$ as a shorthand for ``$\alpha$ is sufficiently large in terms of $\beta_1, \ldots, \beta_M$''.

First, we review a standard coupling that allows us to interpret a polynomial on the slice as a polynomial of independent Rademacher random variables. Let $f$ be a multilinear polynomial of $\vec{\sigma} \sim \Slice(n, k)$ of degree at most $r$:
\[
f(\vec{\sigma}) = \sum_{W \in \binom{[n]}{\le r}} \widehat f(W) \vec{\sigma}^W.
\]
Let $\vec{v} = (v_1(-1), v_1(1), \ldots, v_k(-1), v_k(1))$ be a uniformly random sequence of $2k$ distinct elements of $[n]$, and let $\vec{\xi} = (\xi_1, \ldots, \xi_k)$ be a sequence of independent Rademacher random variables. Then $\{v_1(\xi_1), \ldots, v_k(\xi_k)\}$ is a uniformly random $k$-subset of $[n]$. Interpreting $\vec{\sigma} \sim \Slice(n, k)$ as a function of $\vec{v}$ and $\vec{\xi}$, we can rewrite $f(\vec{\sigma})$ as a polynomial of $\vec{\xi}$ with coefficients depending on $\vec{v}$. Specifically, a calculation identical to \cite[Lemma 3.4]{jain-kwan-mubayi-tran-25} shows that
\begin{equation} \label{eq:def-A}
    f(\vec{\sigma}) = \sum_{I \in \binom{[k]}{\le r}} A_{\vec{v}}(I) \vec{\xi}^{\,I}, \quad \text{ where } \quad
    A_{\vec{v}}(I) = \sum_{W \in \mc{W}_{\vec{v}}(I)} 2^{-|W|}(-1)^{|W \cap \{v_i(-1) : i \in I\}|} \widehat f(W),
\end{equation}
and $\mc{W}_{\vec{v}}(I)$ is the family of subsets of $\{v_1(-1), v_1(1), \ldots, v_k(-1), v_k(1)\}$ of size at most $r$ containing exactly one of $v_i(-1)$ and $v_i(1)$ for each $i \in I$. We further decompose $A_{\vec{v}}(I)$ as $A^{=r}_{\vec{v}}(I) + A^{<r}_{\vec{v}}(I)$, where
\begin{equation} \label{eq:A^{=r}}
    A^{=r}_{\vec{v}}(I) = 2^{-r}\sum_{\substack{W \in \mc{W}_{\vec{v}}(I) \\ |W| = r}} (-1)^{|W \cap \{v_i(-1) : i \in I\}|} \widehat f(W), \quad A^{<r}_{\vec{v}}(I) = \sum_{\substack{W \in \mc{W}_{\vec{v}}(I) \\ |W| \le r-1}} 2^{-|W|}(-1)^{|W \cap \{v_i(-1) : i \in I\}|} \widehat f(W).
\end{equation}
Note that for each $s \in [r]$ with $s \ge |I|$, $\mc{W}_{\vec{v}}(I)$ contains at most $2^s (2k)^{s-|I|}$ sets of size $s$. Therefore, if the coefficients of $f$ satisfy $|\widehat f(W)| \le q$ for each $W \in \binom{[n]}{\le r}$, then for every $I \in \binom{[k]}{\le r}$ we have
\begin{equation} \label{eq:A-bounds}
|A^{=r}_{\vec{v}}(I)| \lesssim_r q k^{r-|I|}, \qquad |A^{<r}_{\vec{v}}(I)| \lesssim_r q k^{r-|I|-1}.
\end{equation}

We also record the following simple corollary of our concentration inequality on the slice (\cref{concentration-on-the-slice}).
\begin{proposition} \label{density-concentration}
    Consider $r, R, n, k_0 \in \NN$ and $\delta > 0$ such that $k_0 \le n \le Rk_0$. Let $\mc{H}$ be an $r$-uniform hypergraph on $n$ vertices with $\delta \binom{n}{r}$ edges, and let $U$ be a uniformly random $k_0$-subset of $[n]$. Then, with probability $1 - \exp(-\Omega_{r, R}(\delta k_0))$, the number of edges in the subhypergraph of $\mc{H}$ induced by the vertices of $U$ lies in the interval $[\frac{\delta}{2}\binom{k_0}{r}, \frac{3\delta}{2}\binom{k_0}{r}]$.
\end{proposition}
\begin{proof}
    Note that the number of edges in the subhypergraph of $\mc{H}$ induced by $U$ can be interpreted as a degree-$r$ polynomial on $\Slice(n, k_0)$, and its expected value is $\delta \binom{k_0}{r}$. So, we can apply \cref{concentration-on-the-slice} with $a_v = |\{S \in E(\mc{H}) : v \in S\}|$. Since
    \[
    \sum_{v \in V(\mc{H})} a_v^2 \le \left(\sum_{v \in V(\mc{H})} a_v\right) \cdot \max_{v \in V(\mc{H})} a_v \le r |E(\mc{H})| \cdot \binom{n-1}{r-1} \le r \delta \binom{n}{r} \binom{n-1}{r-1},
    \]
    we conclude that
    \[
    \mathbb{P}\left[\frac{|\{S \in E(\mc{H}) : S \subseteq U\}|}{\binom{k_0}{r}} \notin [\delta/2, 3\delta/2]\right] \le 2\exp\left(-\frac{\left(\frac{\delta}{2}\binom{k_0}{r}\right)^2}{r \delta \binom{n}{r} \binom{n-1}{r-1}}\right) \le \exp(-\Omega_{r, R}(\delta k_0)). \qedhere
    \]
\end{proof}

Now we prove \cref{sparse-polynomials-general}, which extends \cite[Lemma 5.1]{jain-kwan-mubayi-tran-25} to polynomials with possibly negative coefficients. The proof in \cite{jain-kwan-mubayi-tran-25} does not apply verbatim in this setting, but minor modifications suffice.

\begin{proof}[\textbf{Proof of \cref{sparse-polynomials-general}}]    
    We may assume that $(k/m) \gg (r, q, R)$ (otherwise, one can divide $m$ by a large enough constant depending on $r, q, R$). 

    Let $\vec{v} = (v_1(-1), v_1(1), \ldots, v_k(-1), v_k(1))$ be a uniformly random sequence of $2k$ distinct elements of $[n]$. Also, let
    \[
    \vec{v}(-1) = (v_1(-1), \ldots, v_k(-1)), \quad \vec{v}(1) = (v_1(1), \ldots, v_k(1)),
    \]
    \[
    U(-1) = \{v_1(-1), \ldots, v_k(-1)\}, \quad U(1) = \{v_1(1), \ldots, v_k(1)\}, \quad U = U(-1) \cup U(1).
    \]
    Let $\mc{H}$ be the $r$-uniform hypergraph on vertex set $[n]$ such that $E \subseteq [n]$ is an edge of $\mc{H}$ if and only if $|E| = r$ and $\widehat f(E) \neq 0$. Since $\nu_r(f) \ge m$, $\mc{H}$ has a matching of size $m$. Then, by \cite[Lemma 5.2]{jain-kwan-mubayi-tran-25}, the subhypergraph of $\mc{H}$ induced by the vertices of $\vec{v}(1)$ has a matching of size $m' = \Omega_{r, R}(m)$ with probability $1 - O_{r, R}(1/m)$. Condition on such an outcome of $\vec{v}(1)$, and let $E_1, \ldots, E_{m'}$ be the edges of this matching. Also, we condition on the outcome of the \emph{set} $U(-1)$ (but not on its ordering $\vec{v}(-1)$).

    For a set $S \subseteq U(1)$, let $W(S)$ be a uniformly random set subject to the constraints $S \subseteq W(S) \subseteq U$ and $|W(S)| = r$, and define
    \[
    \boldsymbol{E}(S) = \expecteds{W(S)}{\widehat f(W(S))}.
    \]
    Let $t \gg r$, and take $\delta = 1/(2 q t^r)$. So, by the assumption of the proposition, $\mc{H}$ has at most $\delta \binom{n}{r}$ edges. By \cref{density-concentration} (applied with $k_0 = 2k$), with probability $1 - \exp(-\Omega_{r, q, R}(k))$, the subhypergraph of $\mc{H}$ induced by the vertices of $U$ has at most $\frac{3\delta}{2}\binom{2k}{r}$ edges. From now on, we consider the case when this happens. Then,
    \[
    |\boldsymbol{E}(\emptyset)| \le \frac{q \cdot \frac{3\delta}{2}\binom{2k}{r}}{\binom{2k}{r}} < t^{-r},
    \]
    and $|\boldsymbol{E}(E_i)| \ge 1$ for every $i \in [m']$. Therefore, for every $i \in [m']$, the \emph{minimal} set $F_i \subseteq E_i$ such that $|\boldsymbol{E}(F_i)| \ge t^{|F_i| - r}$ is non-empty. Let $m'' = \lceil m'/r \rceil$. Without loss of generality, we may assume that $F_1, \ldots, F_{m''}$ have the same size $s \in [r]$.

    For every $j \in [m'']$, denote $I_j = \{i \in [k] : v_i(1) \in F_j\}$ and $\vec{V}_j = (v_i(-1))_{i \in I_j}$. Also, let $A_{\vec{v}}(I_j) = A^{=r}_{\vec{v}}(I_j) + A^{<r}_{\vec{v}}(I_j)$ be as in \cref{eq:def-A} and \cref{eq:A^{=r}}.

    \begin{claim} \label{claim:A-large}
        With probability $\Omega_{r, t, q}(1)$ over the choice of $\vec{V}_j$, we have $|A_{\vec{v}}(I_j)| \gtrsim_{r, t} k^{r-s}$.
    \end{claim}
    \begin{claimproof}
        The choice of $\vec{V}_j$ uniquely determines the family of sets $\mc{W}_{\vec{v}}(I_j)$. By \cref{eq:A^{=r}} and \cref{eq:A-bounds}, we have
        \[
        A_{\vec{v}}(I_j) = A^{=r}_{\vec{v}}(I_j) + A^{<r}_{\vec{v}}(I_j) = 2^{-r} \sum_{S \subseteq F_j} (-1)^{s-|S|} \sum_{\substack{W' \in \mc{W}_{\vec{v}}(I_j) \\ |W'| = r, \; W' \cap F_j = S}} \widehat f(W') + O_{r}(q k^{r-s-1}).
        \]
        Note that a uniformly random set from the (random) family of sets $\{W' \in \mc{W}_{\vec{v}}(I_j) : |W'| = r, W' \cap F_j = S\}$ has the same distribution as a uniformly random set $W''$ subject to the constraints $S \subseteq W'' \subseteq U$, $|W''| = r$ and $W'' \cap F_j = S$. Since $1 - O_r(1/k)$ proportion of sets satisfying the first two constraints also satisfy the third one, the expected value of $\widehat f(W'')$ is $\boldsymbol{E}(S) + O_{r}(q/k)$. Therefore,
        \[
        \sum_{\substack{W' \in \mc{W}_{\vec{v}}(I_j) \\ |W'| = r \\ W' \cap F_j = S}} \widehat f(W') = |\{W' \in \mc{W}_{\vec{v}}(I_j) : |W'| = r, \;W' \cap F_j = S\}| \cdot \expecteds{W''}{\widehat f(W'')} = \binom{2k-2s}{r-s} \cdot \boldsymbol{E}(S) + O_r(q k^{r-s-1}).
        \]
        By the choice of $F_j$, we have $|\boldsymbol{E}(F_j)| \ge t^{s-r}$, and $|\boldsymbol{E}(S)| \le t^{s-r-1}$ for every $S \subsetneq F_j$. Hence,
        \[
        \expected{|A_{\vec{v}}(I_j)|} \ge 2^{-r} \binom{2k-2s}{r-s} \left(|\boldsymbol{E}(F_j)| - \sum_{S \subsetneq F_j} |\boldsymbol{E}(S)|\right) + O_r(q k^{r-s-1}) \gtrsim_{r, t} k^{r-s}.
        \]
        On the other hand, by \cref{eq:A-bounds} we have $|A_{\vec{v}}(I_j)| \lesssim_r q k^{r-s}$. Therefore, $|A_{\vec{v}}(I_j)| \gtrsim_{r, t} k^{r-s}$ with probability $\Omega_{r, t, q}(1)$. 
    \end{claimproof}
    By \cref{claim:A-large}, for every $j \in [m'']$, there exists a family $\mc{F}_j$ of ordered $s$-tuples of elements of $U(-1)$ with $|\mc{F}_j| \gtrsim_{r, t, q} k^{s}$, such that if $\vec{V}_j$ falls into $\mc{F}_j$, then $|A_{\vec{v}}(I_j)| \gtrsim_{r, t} k^{r-s}$. Then, by \cite[Lemma 4.4]{jain-kwan-mubayi-tran-25} (proved via a simple application of Chebyshev's inequality), with probability $1 - O_{r, t, q}(1/m'')$ over the choice of $s$-tuples $\vec{V}_1, \ldots, \vec{V}_{m''}$, we have $\vec{V}_j \in \mc{F}_j$ (and thus $|A_{\vec{v}}(I_j)| \gtrsim_{r, t} k^{r-s}$) for $\Omega_{r, t, q}(m'')$ indices $j \in [m'']$.

    We conclude that, with probability $1 - O_{r, t, q, R}(1/m) - \exp(-\Omega_{r, q, R}(k))$ over the randomness of $\vec{v}$, the polynomial $\sum_{I \in \binom{[k]}{\le r}} A_{\vec{v}}(I) \vec{\xi}^{\,I}$ has $\Omega_{r, t, q, R}(m)$ degree-$s$ monomials involving disjoint sets of variables that appear with coefficients of order $k^{r-s}$ (in absolute value). This puts us into the setting of the classical polynomial Littlewood--Offord problem (cf. \cref{def:nu_r,def:LO}). Specifically, we apply \cite[Corollary 2.5]{jain-kwan-mubayi-tran-25} to the randomness of $\vec{\xi}$, and obtain that
    \[
    \prob{f(\vec{\sigma}) = \ell} \le \max_{r' \le r}\LO_{r'}(\Omega_{r, t, q, R}(m)) + O_{r, t, q, R}(1/m) + \exp(-\Omega_{r, q, R}(k)).
    \]
    Recall that $k \ge n/R \ge m/R$, and that $\LO_1(N) = \Theta(1/\sqrt{N})$. Hence, the first summand dominates the other ones, and this completes the proof.
\end{proof}

Next, we turn to the proof of \cref{dense-polynomials}. Jain, Kwan, Mubayi, and Tran~\cite[Lemma 4.1]{jain-kwan-mubayi-tran-25} proved it in the special case when $n = 2k$ and $f$ is homogeneous of degree $r$ with coefficients in $\{0, 1\}$ (that is, when $f$ comes from an $r$-uniform hypergraph). We show how to extend their proof to our slightly more general setting.

\begin{proof}[\textbf{Proof of \cref{dense-polynomials}}]
    We may assume that $(\delta k) \gg (r, q, R)$ (otherwise, the desired bound holds trivially). Also, we may assume that $\delta \le 1/2$.
    Let $\vec{v} = (v_1(-1), v_1(1), \ldots, v_k(-1), v_k(1))$ be a uniformly random sequence of $2k$ distinct elements of $[n]$. Then, $U = \{v_1(-1), v_1(1), \ldots, v_k(-1), v_k(1)\}$ is a uniformly random $2k$-subset of $[n]$. Let $f_U$ be the polynomial obtained from $f$ by setting all variables outside $U$ to zero.
    
    Let $\mc{S} \subseteq \binom{[n]}{r}$ be the largest family of sets such that $|\mc{S}| \le \frac{1}{2}\binom{n}{r}$ and $\widehat f(S_1) \neq \widehat f(S_2)$ for every $S_1 \in \mc{S}$, $S_2 \notin \mc{S}$. Since no $(1-\delta)\binom{n}{r}$ degree-$r$ coefficients of $f$ are equal to the same value, one can check that, in fact, $\delta' = |\mc{S}|/\binom{n}{r}$ satisfies $\delta/2 \le \delta' \le 1/2$ (an identical argument was employed in the proof of \cref{linear-on-the-slice}). Applying \cref{density-concentration} (with $k_0 = 2k$) to the hypergraph with the vertex set $[n]$ and edge set $\mc{S}$, we conclude that, with probability $1 - \exp(-\Omega_{r, R}(\delta k))$ over the choice of $U$, 
    \[
    \frac{|\{S \in \mc{S} : S \subseteq U\}|}{\binom{2k}{r}} \in [\delta'/2, 3\delta'/2] \subseteq [\delta/4, 1-\delta/4].
    \] 
    Condition on such an outcome of the set $U$ (but not on its ordering $\vec{v}$). Then, no $(1-\delta/4)\binom{2k}{r}$ degree-$r$ coefficients of $f_U$ are equal to the same value. In a sense, this reduces the problem to the polynomial $f_U$ on $\Slice(2k, k)$.

    Let $M = (\sum_{W \in \binom{U}{r}} \widehat f(W)) / \binom{2k}{r}$ be the average of degree-$r$ coefficients of $f_U$ (note that for every $W \subseteq U$ we have $\widehat{f}(W) = \widehat{f}_U(W)$). Consider the function $g: \binom{U}{r} \to \RR$ defined by $g(W) = \widehat f(W) - M$, and note that $\sum_{W \in \binom{U}{r}} g(W) = 0$. Also, since at least $\frac{\delta}{4}\binom{2k}{r}$ degree-$r$ coefficients of $f_U$ differ from $M$ by at least $1/2$, we have 
    \[
    \|g\|_1 = \sum_{W \in \binom{U}{r}} |g(W)| \gtrsim_r \delta k^r.
    \]
    For each $s \in [r]$, let $X_s$ be the set of all sequences $\vec{x} = (x_1(-1), x_1(1), \ldots, x_s(-1), x_s(1))$ of $2s$ distinct elements of $U$. For each $\vec{x} \in X_s$, let $\mc{W}_{\vec{x}}^{=r}$ be the family of subsets of $U$ of size $r$ containing exactly one of $x_i(-1)$ and $x_i(1)$ for each $i \in [s]$. As in \cite[Lemma 4.3]{jain-kwan-mubayi-tran-25}, we consider
    \[
    B_s(g) = \sum_{\vec{x} \in X_s} \left|\sum_{W \in \mc{W}^{=r}_{\vec{x}}} (-1)^{|W \cap \{x_1(-1), \ldots, x_s(-1)\}|}g(W)\right|,
    \]
    \[
    b_s(g) = \frac{1}{2k(2k-1)\ldots(2k-(2s-1))} \binom{2k-2s}{r-s}^{-1} B_s(g).
    \]
    Then, \cite[Lemma 9]{bollobas-scott-15} implies that $b_1(g) + \ldots + b_r(g) \gtrsim_r (2k)^{-r} \|g\|_1$, and hence there exists $s \in [r]$ such that  
    \[
    B_s(g) \gtrsim_r k^{r+s} b_s(g) \gtrsim_r k^s \|g\|_1 \gtrsim_r \delta k^{r+s}.
    \]
    Since the value of $B_s(g)$ is unaffected by replacing $g$ with $g + M$, we conclude that
    \[
    \sum_{\vec{x} \in X_s} \left|\sum_{W \in \mc{W}^{=r}_{\vec{x}}} (-1)^{|W \cap \{x_1(-1), \ldots, x_s(-1)\}|}\widehat f(W)\right| \gtrsim_r \delta k^{r+s}.
    \]
    The outer sum here has at most $(2k)^{2s}$ terms, and each of them is $O_{r}(qk^{r-s})$. Therefore, there exists a subset of sequences $\mc{F} \subseteq X_s$, $|\mc{F}| \gtrsim_{r, q}\delta k^{2s}$ such that for every $\vec{x} \in \mc{F}$,
    \begin{equation} \label{eq:delta_k^{r-s}}
    \left|\sum_{W \in \mc{W}^{=r}_{\vec{x}}} (-1)^{|W \cap \{x_1(-1), \ldots, x_s(-1)\}|}\widehat f(W)\right| \gtrsim_{r} \delta k^{r-s}.
    \end{equation}
    Recall that $\vec{v} = (v_1(-1), v_1(1), \ldots, v_k(-1), v_k(1))$ is a uniformly random ordering of the elements of $U$. For every set $I \in \binom{[k]}{\le r}$, let $A_{\vec{v}}(I) = A^{=r}_{\vec{v}}(I) + A^{<r}_{\vec{v}}(I)$ be as in \cref{eq:def-A} and \cref{eq:A^{=r}}.
    Let $m = \lfloor k/s \rfloor \gtrsim_{r} k$. Consider disjoint $s$-subsets $I_1, \ldots, I_{m}$ of $[k]$ and corresponding subsequences of $\vec{v}$: namely, for each $j \in [m]$, we denote
    \[
    \vec{v}_j = (v_{i_1}(-1), v_{i_1}(1), \ldots, v_{i_s}(-1), v_{i_s}(1)), \text{ where } I_j = \{i_1, \ldots, i_s\}.
    \]
    By \cite[Lemma 4.4]{jain-kwan-mubayi-tran-25}, with probability $1 - O_{r, q}(1/(\delta m))$ over the randomness of $\vec{v}$, for $\Omega_{r, q}(\delta m)$ indices $j \in [m]$ we have $\vec{v}_j \in \mc{F}$. By \cref{eq:A^{=r}} and \cref{eq:delta_k^{r-s}}, for every such index $j$ we have $|A^{=r}_{\vec{v}}(I_j)| \gtrsim_{r, q} \delta k^{r-s}$. Since $|A^{<r}_{\vec{v}}(I_j)| \lesssim_{r} q k^{r-s-1}$ by \cref{eq:A-bounds}, we conclude that $|A_{\vec{v}}(I_j)| \gtrsim_{r, q} \delta k^{r-s}$ as well.

    So, with probability $1 - O_{r, q}(1/(\delta k)) - \exp(-\Omega_{r, R}(\delta k))$ over the randomness of $\vec{v}$, the polynomial $\sum\limits_{I \in \binom{[k]}{\le r}} A_{\vec{v}}(I) \vec{\xi}^{\,I}$ has $\Omega_{r, q}(\delta k)$ degree-$s$ monomials involving disjoint sets of variables that appear with coefficients of order $k^{r-s}$ (in absolute value). This puts us into the setting of the classical polynomial Littlewood--Offord problem: applying \cite[Corollary 2.5]{jain-kwan-mubayi-tran-25} to the randomness of $\vec{\xi}$, we obtain that
    \[
    \prob{f(\vec{\sigma}) = \ell} \le \max_{r' \le r}\LO_{r'}(\Omega_{r, q}(\delta k)) + O_{r, q}(1/(\delta k)) + \exp(-\Omega_{r, R}(\delta k)).
    \]
    The first summand dominates the other ones (because $\LO_1(N) = \Theta(1/\sqrt{N})$), and this completes the proof.
\end{proof}

\end{document}